\documentclass{amsart}
\usepackage{latexsym,amsxtra,amscd,ifthen}
\usepackage{amsfonts,cleveref}
\usepackage{verbatim}
\usepackage{amsmath}
\usepackage{amsthm}
\usepackage{amssymb}
\usepackage{adjustbox}
\usepackage{url}
\usepackage[arrow,matrix]{xy}
\usepackage{blkarray, bigstrut}
\usepackage{xparse}
\usepackage{xcolor}

\theoremstyle{plain}

\newtheorem{theorem}{Theorem}
\newtheorem{lemma}[theorem]{Lemma}
\newtheorem{proposition}[theorem]{Proposition}
\newtheorem{corollary}[theorem]{Corollary}

\numberwithin{theorem}{section}
\numberwithin{equation}{theorem}

\theoremstyle{definition}
\newtheorem{definition}[theorem]{Definition}

\newtheorem{example}[theorem]{Example}
\newtheorem{remark}[theorem]{Remark}
\newtheorem{question}[theorem]{Question}
\newtheorem*{question*}{Question}

\DeclareMathOperator{\fpcy}{fpcy}

\DeclareMathOperator{\fpdim}{fpd}
\DeclareMathOperator{\fpc}{fpcx}

\DeclareMathOperator{\fpv}{fpv}

\DeclareMathOperator{\fpgldim}{fpgldim}

\newcommand{\fpk}{\textup{fp}\kappa}

\DeclareMathOperator{\fpg}{fpg}

\DeclareMathOperator{\Vect}{Vect}
\DeclareMathOperator{\End}{End}
\DeclareMathOperator{\Ext}{Ext}

\DeclareMathOperator{\Hom}{Hom}

\DeclareMathOperator{\GKdim}{GKdim}

\DeclareMathOperator{\gldim}{gldim}

\DeclareMathOperator{\Proj}{Proj}
\DeclareMathOperator{\gr}{gr}
\DeclareMathOperator{\Mod}{Mod}

\def\Gr{\operatorname{Gr}}

\usepackage[normalem]{ulem}

\begin{document}

\title[Frobenius-Perron theory]
{Frobenius-Perron theory for projective schemes}

\author{J.M.Chen,  Z.B. Gao,  E. Wicks, J. J. Zhang, X-.H. Zhang and H. Zhu}

\address{Chen: School of Mathematical Science,
Xiamen University, Xiamen, 361005, Fujian, China}

\email{chenjianmin@xmu.edu.cn}

\address{Gao: Department of Communication Engineering,
Xiamen University, Xiamen, 361005, Fujian, China}
\email{gaozhibin@xmu.edu.cn}

\address{Wicks: Department of Mathematics, Box 354350,
University of Washington, Seattle, Washington 98195, USA}

\email{elizabethlwicks@gmail.com}

\address{J.J. Zhang: Department of Mathematics, Box 354350,
University of Washington, Seattle, Washington 98195, USA}

\email{zhang@math.washington.edu}

\address{X.-H. Zhang: College of Sciences, Ningbo University of Technology, Ningb, 315211, 
Zhejiang, China}

\email{zhang-xiaohong@t.shu.edu.cn}

\address{Zhu: Department of Information Sciences, the School of Mathematics and Physics,
Changzhou University,
Changzhou, 213164, Jiangsu, China}

\email{zhuhongazhu@aliyun.com}

\begin{abstract}
The Frobenius-Perron theory of an endofunctor of a $\Bbbk$-linear 
category (recently introduced in \cite{CG}) provides new invariants 
for abelian and triangulated categories. Here we 
study Frobenius-Perron type invariants for derived categories 
of commutative and noncommutative projective schemes. In 
particular, we calculate the Frobenius-Perron dimension for 
domestic and tubular weighted projective lines, define 
Frobenius-Perron generalizations of Calabi-Yau and Kodaira 
dimensions, and provide examples. We apply this theory to 
the derived categories associated to certain Artin-Schelter 
regular and finite-dimensional algebras.
\end{abstract}





\maketitle


\setcounter{section}{-1}
\section{Introduction}
\label{xxsec0}

\bigskip

The Frobenius-Perron dimension of an endofunctor of a category was 
introduced by the authors in \cite{CG}. It can be viewed as a 
generalization of the Frobenius-Perron dimension of an object in
a fusion category introduced by Etingof-Nikshych-Ostrik \cite{ENO2005} 
in early 2000 (also see \cite{EGNO2015, EGO2004, Nik2004}). It is 
shown in \cite{CG} that the Frobenius-Perron dimension of 
either $\Ext^1$ or the suspension of a triangulated category is 
a useful invariant in several different topics such as 
embedding problem, Tame and wild dichotomy, complexity of 
categories. In particular, the Frobenius-Perron invariants have 
strong connections with the representation type of a category 
\cite{CG, ZZ}.

The definition of the Frobenius-Perron dimension of a category 
will be recalled in Section 2. In the present paper we continue 
to develop Frobenius-Perron theory, but we restrict our attention 
to the bounded derived category of coherent sheaves over a 
projective scheme. A projective scheme could be a classical (or 
commutative) one, or a noncommutative one in the sense of \cite{AZ}, 
or a weighted projective line in the sense of \cite{GL}. We refer 
to Section 3 for some basics concerning weighted projective lines. 

Our first goal is to understand the Frobenius-Perron dimension, 
denoted by $\fpdim$, of the bounded derived category of coherent 
sheaves over a weighted projective line, which is also helpful for 
understanding the Frobenius-Perron dimension of the path algebra 
of an acyclic quiver of $\widetilde{A}\widetilde{D}\widetilde{E}$ 
type via the derived equivalence given in Lemma \ref{xxlem2.1}(2). 
Let ${\mathbb X}$ be a weighted projective line (respectively, a 
commutative or noncommutative projective scheme). We use 
$coh({\mathbb X})$ to denote the abelian category of coherent 
sheaves over ${\mathbb X}$ and $D^b(coh({\mathbb X}))$ to denote 
the bounded derived category of $coh({\mathbb X})$. Here is the 
main result in this topic. 

\begin{theorem}[Theorem \ref{xxthm2.13}]
\label{xxthm0.1} Let ${\mathbb X}$ be a weighted projective line
that is either domestic or tubular. Then 
$\fpdim (D^b(coh({\mathbb X})))=1$.
\end{theorem}

Note that Theorem \ref{xxthm0.1} is a ``weighted'' version of 
\cite[Proposition 6.5(1,2)]{CG}. By Lemma \ref{xxlem2.1}(2), we 
obtain the Frobenius-Perron dimension of the bounded derived 
category of finite dimensional representations of acyclic 
quivers of $\widetilde{A}\widetilde{D}\widetilde{E}$ type.

Our second goal is to introduce Frobenius-Perron (``fp'' for short) 
version of some classical invariants. We will focus on fp-analogues 
of two important invariants in projective algebraic geometry, namely, 
\begin{enumerate}
\item[]
Calabi-Yau dimension, and 
\item[]
Kodaira dimension.
\end{enumerate}
Let $\fpk$ (respectively, $\fpcy$) denote the fp version of the 
Kodaira dimension [Definition \ref{xxdef3.5}(1)] (respectively, 
the Calabi-Yau dimension [Definition \ref{xxdef3.1}(3)]). Both are 
defined for bounded derived categories of smooth projective schemes 
or more generally triangulated categories with Serre functor. In 
algebraic geometry, a Calabi-Yau variety has the trivial canonical 
bundle. In noncommutative algebraic geometry, a ``skew Calabi-Yau'' 
scheme may not have a trivial canonical bundle. Our fp version of 
the Calabi-Yau dimension covers the case even when the canonical 
bundle is not trivial. Below is one of the main results in this 
direction.  Note that the definition of $\fpk$ is dependent on a 
chosen structure sheaf.

\begin{theorem}[Propositions \ref{xxpro3.3} and \ref{xxpro3.6}]
\label{xxthm0.2}
Let ${\mathbb X}$ be a smooth projective scheme and  ${\mathcal T}$ 
be the triangulated category $D^b(coh ({\mathbb X}))$ with structure 
sheaf ${\mathcal O}_{{\mathbb X}}$. Then the following hold.
\begin{enumerate}
\item[(1)]
$\fpk({\mathcal T}, {\mathcal O}_{\mathbb X})=\kappa ({\mathbb X})$.
\item[(2)]
$\fpcy({\mathcal T})=\dim {\mathbb X}$. As a consequence, if 
${\mathbb {\mathbb X}}$ is Calabi-Yau, then $\fpcy {\mathcal T}$ 
equals the Calabi-Yau dimension of ${\mathbb X}$.
\end{enumerate}
\end{theorem}

In the noncommutative case we have

\begin{theorem}[Theorem \ref{xxthm4.5}]
\label{xxthm0.3}
Let $A$ be a noetherian connected graded Artin-Schelter 
Gorenstein algebra of injective dimension $d\geq 2$ and 
AS index $\ell$ that is generated in degree 1. Suppose 
that ${\mathbb X}:=\Proj A$ has finite homological 
dimension and that the Hilbert series of $A$ is rational.
Let ${\mathcal T}$ be the bounded derived category of 
$coh({\mathbb X})$ and ${\mathcal A}$ be the image of 
$A$ in $\Proj A$. 
\begin{enumerate}
\item[(1)]
$\fpcy({\mathcal T})=d-1$.
\item[(2)]
If $\ell>0$, then 
$\fpk({\mathcal T},{\mathcal A})=-\infty$ and 
$\fpk^{-1}({\mathcal T},{\mathcal A})=\GKdim A-1$.
\item[(3)]
If $\ell<0$, then 
$\fpk({\mathcal T},{\mathcal A})=\GKdim A-1$ and 
$\fpk^{-1}({\mathcal T},{\mathcal A})=-\infty$.
\item[(4)]
If $\ell=0$, then $\fpk({\mathcal T},{\mathcal A})=
\fpk^{-1}({\mathcal T},{\mathcal A})=0$.
\end{enumerate}
\end{theorem}

Similar results are proved for Piontkovski projective lines, 
see Section 4.

Note that the definitions of fp Calabi-Yau dimension 
and fp Kodaira dimension make sense for bounded derived 
category of left modules over a finite dimension algebra 
of finite global dimension. However, we don't have a 
complete result for that case. Some examples are given in 
Section 5.

When we are working with finite dimensional algebras, it is well-known 
that the global dimension is not a derived invariant. By definition, 
fp Calabi-Yau dimension is a derived invariant. This suggests 
that the fp Calabi-Yau dimension is nicer than the global 
dimension in some aspects. So it is important to study this new 
invariant. To start we ask the following questions.

\begin{question}
\label{xxque0.4} Let $A$ be a finite dimensional algebra of finite
global dimension and let ${\mathcal T}$ be the derived category
$D^b(\Mod_{f.d}-A)$.
\begin{enumerate}
\item[(1)]
Is $\fpcy({\mathcal T})$ always finite?
\item[(2)]
What is the set of possible values of $\fpcy({\mathcal T})$ when 
$A$ varies?
\item[(3)]
What is the set of possible values of $\fpk({\mathcal T},A)$ when 
$A$ varies?
\end{enumerate}
\end{question}

More questions and some partial answers are given in Section 5. 
Some other examples are given in \cite{Wi, ZZ}.

We have outlined several important applications of Frobenius-Perron 
invariants in \cite{CG}. Next we would like to mention one
surprising application of the Frobenius-Perron curvature defined in
\cite[Definition 2.3(4)]{CG}.

\begin{theorem} \cite[Corollary 0.6]{ZZ} 
\label{xxthm0.5}
Suppose that the bounded derived categories of representations of
two finite acyclic quivers are equivalent as tensor triangulated 
categories. Then the quivers are isomorphic.
\end{theorem}

This result is striking because it is well-known that, for 
two Dynkin quivers with the same underlying Dynkin diagram, their 
derived categories are triangulated equivalent \cite{BGP, Ha}, 
even if the quivers are non-isomorphic. We hope that different 
Frobenius-Perron invariants will  become effective tools in the 
study of triangulated categories and monoidal (or tensor) 
triangulated categories.

This paper is organized as follows. Some background material are 
provided in Section \ref{xxsec1}. In Section \ref{xxsec2}, we 
review some facts about weighted projective lines and prove 
Theorem \ref{xxthm0.1}. In Section \ref{xxsec3}, we introduce 
fp-version of Calabi-Yau dimension and Kodaira dimension for 
Ext-finite triangulated categories with Serre functor and prove 
Theorem \ref{xxthm0.2}. In Section \ref{xxsec4}, fp Calabi-Yau 
dimension and fp Kodaira dimension are studied for noncommutative 
projective schemes and Theorem \ref{xxthm0.3} is proved there. 
In Section \ref{xxsec5}, some partial results, comments and 
examples are given concerning finite dimensional algebras. 
Sections \ref{xxsec6} and \ref{xxsec7} are appendices. The proof 
of Theorem \ref{xxthm0.1} is dependent on some linear algebra 
computation given in Section \ref{xxsec6}. This paper can be 
viewed as a sequel of \cite{CG}.

\section{Preliminaries and definitions}
\label{xxsec1}

Throughout let $\Bbbk$ be a base field that is algebraically closed.
Let everything be over $\Bbbk$.

We are mainly interested in the derived category $D^b(coh({\mathbb X}))$
where ${\mathbb X}$ is a smooth commutative or noncommutative projective 
scheme, but most definitions work for more general pre-triangulated 
(or abelian) categories.

Part of this section is copied from \cite{CG}.

\subsection{Spectral radius of a square matrix}
\label{xxsec1.1}
Let $A$ be an $n\times n$-matrix over complex numbers ${\mathbb C}$.
The {\it spectral radius} of $A$ is defined to be 
$$\rho(A)=\max\{ |r_1|, |r_2|, \cdots, |r_n|\}$$
where $\{r_1,r_2,\cdots, r_n\}$ is the complete set of eigenvalues 
of $A$. When each entry of $A$ is a positive real number, $\rho(A)$ 
is also called the {\it Perron root} or the {\it Perron-Frobenius 
eigenvalue} of $A$.

In order to include the ``infinite-dimensional'' setting, we extend 
the definition of the spectral radius in the following way.

Let $A=(a_{ij})_{n\times n}$ be an $n\times n$-matrix with 
entries $a_{ij}$ in ${\mathbb R}^{+}:=
{\mathbb R}\cup \{\pm \infty\}$. 
Define $A'=(a'_{ij})_{n\times n}$ where 
$$a'_{ij}=\begin{cases} a_{ij} & a_{ij}\neq \pm \infty,\\
x_{ij} & a_{ij}=\infty,\\
-x_{ij} & a_{ij}=-\infty.
\end{cases}
$$
In other words, we are replacing $\infty$  in the
$(i,j)$-entry by a finite real number, called $x_{ij}$, in the
$(i,j)$-entry. Or $x_{ij}$ are considered as function or a variable 
mapping ${\mathbb R}\to {\mathbb R}$. 

\begin{definition}
\label{xxdef1.1} 
Let $A$ be an $n\times n$-matrix with entries in ${\mathbb R}^{+}$.
The {\it spectral radius} of $A$ is defined to be
\begin{equation}
\notag
\rho(A):=\liminf_{{\text{all}}\; x_{ij}\to \infty} \; \rho(A').
\end{equation}
\end{definition}

See \cite[Remark 1.3 and Example 1.4]{CG}.

\subsection{Frobenius-Perron dimension of a quiver}
\label{xxsec1.2}

\begin{definition}
\label{xxdef1.2} \cite[Definition 1.6]{CG} 
Let $Q$ be a quiver.
\begin{enumerate}
\item[(1)]
If $Q$ has finitely many vertices, then the 
{\it Frobenius-Perron dimension} of 
$Q$ is defined to be 
$$\fpdim Q:=\rho(A(Q))$$ 
where $A(Q)$ is the adjacency matrix of $Q$.
\item[(2)]
Let $Q$ be any quiver. 
The {\it Frobenius-Perron dimension} of $Q$ is defined to 
be 
$$\fpdim Q:=\sup\{ \fpdim Q'\}$$
where $Q'$ runs over all finite subquivers of $Q$.
\end{enumerate}
\end{definition}

\subsection{Frobenius-Perron dimension of an endofunctor}
\label{xxsec1.3}
Let ${\mathcal C}$ denote a $\Bbbk$-linear category.  For 
simplicity, we use $\dim(A,B)$ for $\dim \Hom_{\mathcal C}(A,B)$
for any two objects  $A$ and $B$ in ${\mathcal C}$. Here the
second $\dim$ is $\dim_{\Bbbk}$.

The set of finite subsets of nonzero objects in ${\mathcal C}$ 
is denoted by $\Phi$ and the set of subsets of $n$ nonzero objects 
in ${\mathcal C}$ is denoted by $\Phi_n$ for each $n\geq 1$.
It is clear that $\Phi=\bigcup_{n\geq 1} \Phi_n$. We do not
consider the empty set as an element of $\Phi$. 

\begin{definition}\cite[Definition 2.1]{CG}
\label{xxdef1.3} 
Let $\phi:=\{X_1, X_2, \cdots,X_n\}$ be a finite subset of nonzero
objects in ${\mathcal C}$, namely, $\phi\in \Phi_n$. Let $\sigma$ 
be an endofunctor of ${\mathcal C}$.
\begin{enumerate}
\item[(1)]
The {\it adjacency matrix} of $(\phi, \sigma)$ is defined to be
$$A(\phi, \sigma):=(a_{ij})_{n\times n}, \quad 
{\text{where}}\;\; a_{ij}:=\dim(X_i, \sigma(X_j)) \;\;\forall i,j.$$
\item[(2)]
An object $M$ in ${\mathcal C}$ is called a {\it brick} 
\cite[Definition 2.4, Ch. VII]{ASS} if 
\begin{equation}
\notag
\Hom_{\mathcal C}(M,M)=\Bbbk. 
\end{equation}
If ${\mathcal C}$ is a pre-triangulated category 
\cite[Definition 1.1.2]{Ne} with suspension $\Sigma$, an object 
$M$ in ${\mathcal C}$ is called an {\it atomic} object if it 
is a brick and satisfies
\begin{equation}
\notag
\Hom_{\mathcal C}(M, \Sigma^{-i}(M))=0, \quad \forall \; i>0.
\end{equation}
\item[(3)]
$\phi\in \Phi$ is called a {\it brick set} (respectively, an 
{\it atomic set})  if each $X_i$ is a brick (respectively, atomic) 
and 
$$\dim(X_i, X_j)=\delta_{ij}$$
for all $1\leq i,j\leq n$. The set of brick (respectively, atomic) 
$n$-object subsets is denoted by $\Phi_{n,b}$ (respectively, 
$\Phi_{n,a}$). We write $\Phi_{b}=\bigcup_{n\geq 1} 
\Phi_{n,b}$ (respectively, $\Phi_{a}=\bigcup_{n\geq 1} 
\Phi_{n,a}$). 
\end{enumerate}
\end{definition}

\begin{definition} \cite[Definition 2.3]{CG}
\label{xxdef1.4}
Retain the notation as in Definition \ref{xxdef1.3}, and 
we use $\Phi_{b}$ as the testing objects. When ${\mathcal C}$ 
is a pre-triangulated category, $\Phi_{b}$ is automatically 
replaced by $\Phi_{a}$ unless otherwise stated.
\begin{enumerate}
\item[(1)]
The {\it $n$th Frobenius-Perron dimension} of $\sigma$ is defined to be
$$\fpdim^n (\sigma):=\sup_{\phi\in \Phi_{n,b}}\{\rho(A(\phi,\sigma))\}.$$
If $\Phi_{n,b}$ is empty, then, by convention, $\fpdim^n(\sigma)=0$.
\item[(2)]
The {\it Frobenius-Perron dimension} of $\sigma$ is defined to be
$$\fpdim (\sigma):=\sup_n \{\fpdim^n(\sigma)\}
=\sup_{\phi\in \Phi_{b}} \{\rho(A(\phi,\sigma)) \}.$$
\item[(3)]
The {\it Frobenius-Perron growth} of $\sigma$ is defined to be
$$\fpg (\sigma):=\sup_{\phi\in \Phi_{b}} 
\{\limsup_{n\to\infty} \; \log_{n}(\rho(A(\phi,\sigma^n))) \}.$$
By convention, $\log_n 0 =-\infty$.
\item[(4)]
The {\it Frobenius-Perron curvature} of $\sigma$ is defined to be
$$\fpv (\sigma):=\sup_{\phi\in \Phi_{b}} \{\limsup_{n\to\infty} \;  
(\rho(A(\phi,\sigma^n)))^{1/n} \}.$$
\end{enumerate}
\end{definition}

In this paper, we only use $\Phi_{b}$ and $\Phi_{a}$ as the testing 
objects. But in principal one can use other testing objects, see 
Section \ref{xxsec7}. We continue to review definitions from 
\cite{CG}.

\begin{definition} \cite[Definition 2.7]{CG}
\label{xxdef1.5} 
\begin{enumerate}
\item[(1)]
Let ${\mathfrak A}$ be an abelian category. The 
{\it Frobenius-Perron dimension} of ${\mathfrak A}$
is defined to be
$$\fpdim {\mathfrak A}:=\fpdim (E^1)$$
where $E^1:=\Ext^1_{\mathfrak A}(-,-)$
is defined as in \cite[Example 2.6(1)]{CG}. The 
{\it Frobenius-Perron theory} of ${\mathfrak A}$ is the collection
$$\{\fpdim^m (E^n)\}_{m\geq 1, n\geq 0}$$
where $E^n:=\Ext^n_{\mathcal A}(-,-)$ is defined as in 
\cite[Example 2.6(1)]{CG}. 
\item[(2)]
Let ${\mathcal T}$ be a pre-triangulated category with suspension
$\Sigma$. The 
{\it Frobenius-Perron dimension} of ${\mathcal T}$
is defined to be
$$\fpdim {\mathcal T}:=\fpdim (\Sigma).$$
\item[(3)]
The {\it fp-global dimension} of ${\mathcal T}$ is defined to be
$$\fpgldim {\mathcal T}:=\sup \{n \mid \fpdim(\Sigma^n)\neq 0\}.$$
\end{enumerate}
\end{definition}

Fix an endofunctor $\sigma$ of a category ${\mathcal C}$. For 
a set of bricks $B$ in ${\mathcal C}$ (or a set of atomic 
objects when ${\mathcal C}$ is triangulated), we define 
$$\fpdim^n\mid_{B} (\sigma) =\sup\{ \rho(A(\phi,\sigma)) \;
\mid \; \phi:=\{X_1,\cdots,X_n\}\in \Phi_{n,b}, 
\;\; {\rm{and}} \;\; X_i\in B\;\; \forall i\}.$$

Let $\Lambda:=\{\lambda\}$ be a totally ordered set. We say 
a set of bricks $B$ in ${\mathcal C}$ has a 
{\it $\sigma$-decomposition} $\{B^{\lambda}\}_{\lambda\in \Lambda}$
(based on $\Lambda$) if the following holds.
\begin{enumerate}
\item[(1)]
$B$ is a disjoint union $\bigcup_{\lambda\in \Lambda} B^{\lambda}$.
\item[(2)]
If $X\in B^{\lambda}$ and $Y\in B^{\delta}$
with $\lambda<\delta$, $\Hom_{\mathcal C}(X,\sigma(Y))=0$.
\end{enumerate}

The following is \cite[Lemma 6.1]{CG}.

\begin{lemma}\cite[Lemma 6.1]{CG}
\label{xxlem1.6} Let $n$ be a positive integer.
Suppose that $B$ has a 
{\it $\sigma$-decomposition} $\{B^{\lambda}\}_{\lambda\in \Lambda}$.
Then 
$$\fpdim^n|_{B}(\sigma)\leq \sup_{\lambda \in \Lambda, m\leq n} 
\{\fpdim^m|_{B^{\lambda}}(\sigma)\}.$$
\end{lemma}

\section{Frobenius-Perron theory of weighted projective lines}
\label{xxsec2}

The main goal of this section is to recall some facts about 
weighted projective lines and then to prove Theorem \ref{xxthm0.1}.
 
\subsection{Weighted projective lines}
\label{xxsec2.1}
First we recall the definition and some basics about weighted 
projective lines. Details can be found in \cite[Section 1]{GL}.

For $t\geq 1$, let ${\bf p}:=(p_0,p_1,\cdots,p_t)$ be a 
$(t+1)$-tuple of positive integers, called the 
{\it weight sequence}. Let
${\bf D}:=(\lambda_0, \lambda_1,\cdots, \lambda_t)$ be a 
sequence of distinct points of the projective line 
${\mathbb P}^1$ over $\Bbbk$. We normalize ${\bf D}$ 
so that $\lambda_0=\infty$, $\lambda_1=0$ and 
$\lambda_2=1$ (if $t\geq 2$). Let
$$S:=\Bbbk[X_0,X_1,\cdots,X_t]/(X_i^{p_i}-X_1^{p_1}+\lambda_i X_0^{p_0},
i=2,\cdots,t).$$
The image of $X_i$ in $S$ is denoted by $x_i$ for all $i$.
Let ${\mathbb L}$ be the abelian group of rank 1 
generated by $\overrightarrow{x_i}$ for $i=0,1,\cdots,t$
and subject to the relations
$$p_0 \overrightarrow{x_0}= \cdots =p_i \overrightarrow{x_i}=\cdots
=p_t \overrightarrow{x_t}=: \overrightarrow{c}.$$
The algebra $S$ is ${\mathbb L}$-graded by setting $\deg x_i=
\overrightarrow{x_i}$. The corresponding 
{\it weighted projective line}, 
denoted by ${\mathbb X}({\bf p},{\bf D})$ or simply ${\mathbb X}$,
is a noncommutative space whose category of coherent sheaves is 
given by the quotient category 
$$coh({\mathbb X}):=\frac{\gr^{\mathbb L}-S}{\gr_{f.d.}^{\mathbb L}-S}.$$

The weighted projective lines are classified into the following
three classes:
\begin{equation}
\notag
{\mathbb X} \;\; {\rm{is}}\;\; 
\begin{cases} domestic \;\; & {\rm{if}} \;\; {\bf p} 
\;\; {\rm{is}}\; (p, q), (2,2,n), (2,3,3), (2,3,4), (2,3,5);\\
tubular \;\; & {\rm{if}} \;\; {\bf p} 
\;\; {\rm{is}}\; (2,3,6), (3,3,3), (2,4,4), (2,2,2,2);\\
wild \;\; & {\rm{otherwise}}.
\end{cases}
\end{equation}
In \cite[Section 4.4]{Sc}, domestic (respectively, tubular,
wild) weighted projective lines are called {\it parabolic} 
(respectively, {\it elliptic, hyperbolic}). Let ${\mathbb X}$ 
be a weighted projective line. A sheaf $F\in coh({\mathbb X})$
is called {\it torsion} if it is of finite length in 
$coh({\mathbb X})$. Let $Tor({\mathbb X})$ denote the full
subcategory of $coh({\mathbb X})$ consisting of all torsion
objects. By \cite[Lemma 4.16]{Sc}, the category $Tor({\mathbb X})$
decomposes as a direct product of orthogonal blocks
\begin{equation}
\label{E2.0.1}\tag{E2.0.1}
Tor({\mathbb X})=\prod_{x\in {\mathbb P}^1\setminus
\{\lambda_0,\lambda_1,\cdots,\lambda_{t}\}} Tor_{x}
\; \times \; \prod_{i=0}^{t} Tor_{\lambda_i}
\end{equation}
where $Tor_{x}$ is equivalent to the category of nilpotent
representations of the Jordan quiver (with one vertex and 
one arrow) over the residue field $\Bbbk_{x}$ and where
$Tor_{\lambda_i}$ is equivalent to the category of nilpotent 
representations over $\Bbbk$ of the cyclic quiver of length
$p_i$, see Example \ref{xxex2.3}. A simple object in 
$coh({\mathbb X})$ is called {\it ordinary simple} (see 
\cite{GL}) if it is the skyscraper sheaf ${\mathcal O}_x$ of 
a closed point $x\in {\mathbb P}^1\setminus
\{\lambda_0,\lambda_1,\cdots,\lambda_{t}\}$.

Let $Vect({\mathbb X})$ be the full subcategory of 
$coh({\mathbb X})$ consisting of all vector bundles. Similar
to the elliptic curve case \cite[Section 4]{BB}, one can 
define the concepts of {\it degree}, {\it rank} and 
{\it slope} of a vector bundle on a weighted projective 
line ${\mathbb X}$; details are given in \cite[Section 4.7]{Sc} 
and \cite[Section 2]{LM}.  For each 
$\mu\in {\mathbb Q}$, let $Vect_{\mu}({\mathbb X})$ be the 
full subcategory of $Vect({\mathbb X})$ consisting of all 
semistable vector bundles of slope $\mu$. By convention, 
$Vect_{\infty}({\mathbb X})$ denotes $Tor({\mathbb X})$.
By \cite[Comments after Corollary 4.34]{Sc}, when 
${\mathbb X}$ is a domestic or tubular weighted projective
line, every indecomposable object in $coh({\mathbb X})$ is in 
\begin{equation}
\notag
\bigcup_{\mu\in {\mathbb Q}\cup\{\infty\}} Vect_{\mu}({\mathbb X}).
\end{equation}

Below we collect some nice properties of weighted projective lines.
The definition of a stable tube (or simply tube) was introduced 
in \cite{Ri}.

\begin{lemma} \cite[Lemma 7.9]{CG}
\label{xxlem2.1}
Let ${\mathbb X}={\mathbb X}({\bf p}, {\bf D})$ be a weighted 
projective line.
\begin{enumerate}
\item[(1)]
$coh({\mathbb X})$ is noetherian and hereditary.
\item[(2)]
$$D^b(coh({\mathbb X})) \cong 
\begin{cases} 
D^b(\Mod_{f.d.}-\Bbbk \widetilde{A}_{p, q}) & {\rm{if}}\;\; {\bf p}=(p,q),\\
D^b(\Mod_{f.d.}-\Bbbk \widetilde{D}_n) & {\rm{if}}\;\; {\bf p}=(2,2,n),\\
D^b(\Mod_{f.d.}-\Bbbk \widetilde{E}_6) & {\rm{if}}\;\; {\bf p}=(2,3,3),\\
D^b(\Mod_{f.d.}-\Bbbk \widetilde{E}_7) & {\rm{if}}\;\; {\bf p}=(2,3,4),\\
D^b(\Mod_{f.d.}-\Bbbk \widetilde{E}_8) & {\rm{if}}\;\; {\bf p}=(2,3,5).
\end{cases}
$$
\item[(3)]
Let ${\mathcal S}$ be an ordinary simple 
object in $coh({\mathbb X})$.
Then $\Ext^1_{\mathbb X}({\mathcal S},{\mathcal S})=\Bbbk$.
\item[(4)]
$\fpdim^1 (coh({\mathbb X}))\geq 1$.
\item[(5)]
If ${\mathbb X}$ is tubular or domestic, then $\Ext^1_{\mathbb X}(X,Y)=0$ 
for all $X\in Vect_{\mu'}({\mathbb X})$ and $Y\in Vect_{\mu}({\mathbb X})$
with $\mu'< \mu$.
\item[(6)]
If ${\mathbb X}$ is domestic, then $\Ext^1_{\mathbb X}(X,Y)=0$ 
for all $X\in Vect_{\mu'}({\mathbb X})$ and $Y\in Vect_{\mu}({\mathbb X})$
with $\mu'\leq \mu<\infty$. As a consequence, 
$\fpdim(\Sigma\mid_{Vect_{\mu}({\mathbb X})})=0$
for all $\mu<\infty$.
\item[(7)]
Suppose ${\mathbb X}$ is tubular or domestic. 
Then every indecomposable vector bundle ${\mathbb X}$ 
is semi-stable.
\item[(8)]
Suppose ${\mathbb X}$ is tubular
and let $\mu\in {\mathbb Q}$. 
Then each $Vect_{\mu}({\mathbb X})$ is a uniserial category. 
Accordingly indecomposables in $Vect_{\mu}({\mathbb X})$ 
decomposes into Auslander-Reiten components, which all are 
stable tubes of finite rank. In fact, for every 
$\mu\in {\mathbb Q}$, $$Vect_{\mu}({\mathbb X})\cong 
Vect_{\infty}({\mathbb X})=Tor({\mathbb X}).$$ 
\end{enumerate}
\end{lemma}

\begin{proof} (1) This is well-known, see \cite[Theorem 2.2]{Le}

(2) See \cite[Proposition 5.1(i)]{KLM} and \cite[5.4.1]{GL}.

(3) Let ${\mathcal S}$ be an ordinary simple object which is of the form 
${\mathcal O}_x$ for some $x\in {\mathbb P}^1\setminus 
\{\lambda_0,\cdots,\lambda_t\}$. 
Then ${\mathcal S}$ is a brick and $\Ext^1_{\mathbb X}({\mathcal S},{\mathcal S})=
\Ext^1_{\mathbb X}({\mathcal O}_x, {\mathcal O}_x)=\Bbbk$. 

(4) Follows from (3) by taking $\phi:=\{{\mathcal S}\}$.

(5) This is \cite[Corollary 4.34(i)]{Sc}.

(6) This is \cite[Comments after Corollary 4.34]{Sc}. 
The consequence is clear.

(7) \cite[Theorem 5.6(i)]{GL}.

(8) See \cite[Theorem 4.42]{Sc} and \cite[Theorem 5.6(iii)]{GL}.
\end{proof}

Our main goal in this section is to prove Theorem \ref{xxthm0.1}.
Eventually one should ask the following question.

\begin{question} \cite[Question 7.11]{CG}
\label{xxque2.2}
Let ${\mathbb X}$  be a weighted projective line. What is 
the exact value of $\fpdim^n \; D^b(coh({\mathbb X}))$, for 
$n\geq 1$, in terms of other invariants of ${\mathbb X}$?
\end{question}

\subsection{Standard stable tubes \cite{LS, Ri}}
\label{xxsec2.2}
In this subsection we would like to understand the 
(standard) stable tubes in $Tor_{\lambda_i}$ in the 
decomposition \eqref{E2.0.1}, which is the Auslander-Reiten 
quiver of the $x$-nilpotent (or $x$-torsion) 
representations of the algebra in the following example.

\begin{example}
\label{xxex2.3}
Let $\xi$ be a primitive $n$th root of unity.
Let $T_n$ be the algebra 
$$T_n:=\frac{\Bbbk \langle g,x\rangle}{(g^n-1, gx-\xi xg)}.$$
This algebra can be expressed by using a group action.
Let $G$ be the group 
$$\{g\mid g^n=1\}\cong {\mathbb Z}/(n)$$ 
acting on the polynomial ring $\Bbbk[x]$ by $g\cdot x=\xi x$.
Then $T_n$ is naturally isomorphic to the skew group ring 
$\Bbbk[x]\ast G$. Let $\overrightarrow{A_{n-1}}$ denote the 
cycle quiver with $n$ vertices, namely, the quiver with one 
oriented cycle connecting $n$ vertices. It is also known that 
$T_n$ is isomorphic to the path algebra of the quiver 
$\overrightarrow{A_{n-1}}$.
\end{example}

Let ${\mathfrak A}$ be the category of finite dimensional
left $T_n$-modules that are $x$-torsion. 
In this subsection we will show that $\fpdim ({\mathfrak A})=1$
[Corollary \ref{xxcor2.12}]. We start with somewhat more 
general setting. 

Let $A$ be any algebra and let $\Mod_{f.d.}-A$ be the category 
of finite dimensional left $A$-modules. Let $\Gamma(\Mod_{f.d.}-A)$
denote the Auslander-Reiten quiver with Auslander-Reiten translation 
$\tau$.

Let $\mathfrak{C}$ be a component of $\Gamma(\Mod_{f.d.}-A)$. 
We say $\mathfrak{C}$ is a \emph{self-hereditary component} of 
$\Gamma(\Mod_{f.d.}-A)$ if for each pair of indecomposable 
$A$-modules $X$ and $Y$ in $\mathfrak{C}$, we have 
$\Ext^{2}_{A}(X, Y)=0$.

We now recall some results from the book \cite{SS}. The 
definitions can be found in \cite{SS}. Let 
$\phi:=\{E_{1}, \cdots ,E_{r}\}$ be a brick set in 
$\mathfrak{C}$. (In \cite{SS}, $\phi$ is called a finite 
family of pairwise orthogonal bricks.) The 
\emph{extension category} \cite[p.13]{SS} of $\phi$, denoted 
by $\mathcal{E}$, $\mathcal{E}_{A}$, or 
$\mathcal{EXT}_{A}(E_{1}, \cdots ,E_{r})$, is defined to be 
the full subcategory of $\Mod_{f.d.}-A$ whose nonzero objects 
are all the objects $M$ such that there exists a chain of
submodules 
$$M = M_{0} \supsetneq M_{1} \supsetneq \cdots \supsetneq M_{l}= 0,$$ 
for some $l\geq 1$, with $M_{i}/M_{i+1}$ isomorphic to one of the 
bricks $E_{1}, \cdots ,E_{r}$ for all $0\leq i <l$. We say 
$\{E_1, \cdots, E_r\}$ is a $\tau$-cycle if $\tau(E_i)=E_{i-1}$ 
for all $i\in {\mathbb Z}/(r)$.

We will be using the notation introduced in \cite{SS}. For example, 
$E_{i}[j]$ represents some uniserial object, which is 
nothing to do with the $j$th suspension of $E_{i}$.

\begin{theorem} \cite[Lemma 2.1 and Theorem 2.2 in Ch. X]{SS}
\label{xxthm2.4}
Let $\phi:=\{E_{1}, \cdots ,E_{r}\}$, with $r\geq 1$, 
be a brick set on $\Mod_{f.d.}-A$. Suppose that $\phi$ is a 
$\tau$-cycle and a self-hereditary family \cite[p.14]{SS}.
Then the extension category ${\mathcal E}$ is an abelian category 
with the following properties.
\begin{enumerate}
\item[(1)] 
For each pair $(i, j)$, with $1\leq i\leq r$ and $j\geq 1$, 
there exist a uniserial object $E_{i}[j]$ of $\mathcal{E}$-length 
$l_{\mathcal{E}}(E_{i}[j]) = j$ in the category $\mathcal{E}$, and
homomorphisms
$$u_{ij} : E_{i}[j-1]\longrightarrow E_{i}[j],\quad {\text{and}}
\quad 
p_{ij}: E_{i}[j]\longrightarrow E_{i+1}[j-1],$$
for $j\geq 2$, such that we have two short exact sequences in 
$\Mod_{f.d.}-A$
$$0\longrightarrow E_{i}[j-1]\xrightarrow{\;\; u_{ij}\;\; }E_{i}[j] 
\xrightarrow{\;\; p'_{ij}\;\; } E_{i+j-1}[1]\longrightarrow 0,$$
$$0\longrightarrow E_{i}[1]\xrightarrow{\;\; u'_{ij}\;\; } E_{i}[j] 
\xrightarrow{\;\; p_{ij}\;\; } E_{i+1}[j-1]\longrightarrow 0,$$
where $p'_{ij}=p_{i+j-2,2}\circ \cdots \circ p_{ij}$ and 
$u'_{ij}=u_{ij}\circ \cdots \circ u_{i2}$. 
Moreover,  for each $j\geq 2$, there exists an almost split sequence
$$0\rightarrow E_{i}[j-1]\xrightarrow{\left\lbrack
\begin{array}{c}
p_{i,j-1}\\
u_{ij}
\end{array}
\right\rbrack} E_{i+1}[j-2]\oplus E_{i}[j]\xrightarrow{(u_{i+1,j-1} 
\ p_{ij})} E_{i+1}[j-1]\rightarrow 0,$$
in $\Mod_{f.d.}-A$, where we set $E_{i}[0] = 0$ and $E_{i+kr}[m] = E_{i}[m]$, 
for $m\geq 1$ and all $k\in \mathbb{Z}$.
\item[(2)] 
The indecomposable uniserial objects $E_{i}[j]$, with 
$i \in \{1, \cdots , r\}$ and $j\geq 1$, of the category 
$\mathcal{E}$, connected by the homomorphisms $u_{ij}:
E_{i}[j-1]\rightarrow E_{i}[j]$ and 
$p_{ij}: E_{i}[j]\rightarrow E_{i+1}[j-1]$, form the infinite
diagram presented below.

\bigskip

\vspace*{12mm}

\bigskip

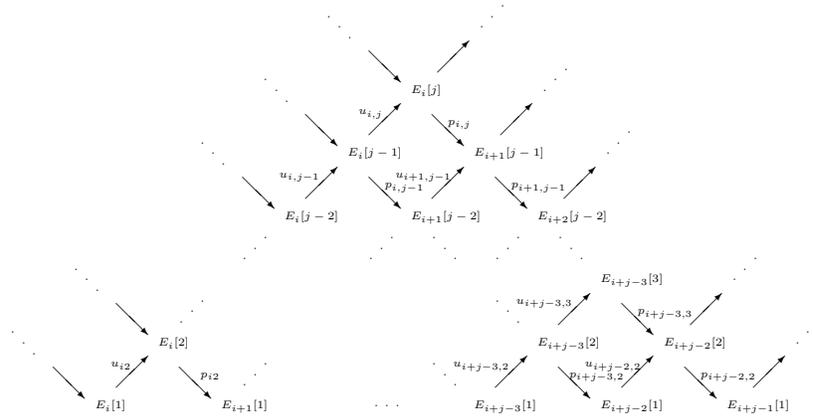
\begin{centering}
\adjustbox{scale=.7}{
{\tiny
\unitlength=1mm
\begin{picture}(140,60)(-15,0)
\put(4,4){\vector(1,1){6}}
\put(40,40){\vector(1,1){6}}
\put(52,52){\vector(1,1){6}}
\put(16,8){\vector(1,-1){6}}
\put(52,44){\vector(1,-1){6}}
\put(64,56){\vector(1,-1){6}}
\put(64,40){\vector(1,1){6}}
\put(76,44){\vector(1,-1){6}}
\put(76,4){\vector(1,1){6}}
\put(88,16){\vector(1,1){6}}
\put(88,8){\vector(1,-1){6}}
\put(100,20){\vector(1,-1){6}}
\put(100,4){\vector(1,1){6}}
\put(112,8){\vector(1,-1){6}}
\put(16,16){$\cdot$}
\put(18,18){$\cdot$}
\put(20,20){$\cdot$}
\put(28,28){$\cdot$}
\put(30,30){$\cdot$}
\put(32,32){$\cdot$}
\put(28,4){$\cdot$}
\put(30,6){$\cdot$}
\put(32,8){$\cdot$}
\put(52,28){$\cdot$}
\put(54,30){$\cdot$}
\put(56,32){$\cdot$}
\put(64,32){$\cdot$}
\put(66,30){$\cdot$}
\put(68,28){$\cdot$}
\put(76,28){$\cdot$}
\put(78,30){$\cdot$}
\put(80,32){$\cdot$}
\put(64,8){$\cdot$}
\put(66,6){$\cdot$}
\put(68,4){$\cdot$}
\put(76,20){$\cdot$}
\put(78,18){$\cdot$}
\put(80,16){$\cdot$}
\put(88,32){$\cdot$}
\put(90,30){$\cdot$}
\put(92,28){$\cdot$}
\put(0,0){$E_{i}[1]$}
\put(12,12){$E_{i}[2]$}
\put(36,36){$E_{i}[j-2]$}
\put(48,48){$E_{i}[j-1]$}
\put(60,60){$E_{i}[j]$}
\put(24,0){$E_{i+1}[1]$}
\put(60,36){$E_{i+1}[j-2]$}
\put(72,48){$E_{i+1}[j-1]$}
\put(84,36){$E_{i+2}[j-2]$}
\put(72,0){$E_{i+j-3}[1]$}
\put(84,12){$E_{i+j-3}[2]$}
\put(96,24){$E_{i+j-3}[3]$}
\put(96,0){$E_{i+j-2}[1]$}
\put(108,12){$E_{i+j-2}[2]$}
\put(120,0){$E_{i+j-1}[1]$}

\put(3,8){\tiny{$u_{i2}$}}
\put(20,6){\tiny{$p_{i2}$}}
\put(68,8){\tiny{$u_{i+j-3,2}$}}
\put(90,6){\tiny{$p_{i+j-3,2}$}}
\put(93,8){\tiny{$u_{i+j-2,2}$}}
\put(115,6){\tiny{$p_{i+j-2,2}$}}
\put(80,20){\tiny{$u_{i+j-3,3}$}}
\put(103,18){\tiny{$p_{i+j-3,3}$}}
\put(35,44){\tiny{$u_{i,j-1}$}}
\put(55,42){\tiny{$p_{i,j-1}$}}
\put(57,44){\tiny{$u_{i+1,j-1}$}}
\put(79,42){\tiny{$p_{i+1,j-1}$}}
\put(50,56){\tiny{$u_{i,j}$}}
\put(67,54){\tiny{$p_{i,j}$}}

\put(-8,8){\vector(1,-1){6}}
\put(-16,14){$\cdot$}
\put(-14,12){$\cdot$}
\put(-12,10){$\cdot$}

\put(4,20){\vector(1,-1){6}}
\put(-4,26){$\cdot$}
\put(-2,24){$\cdot$}
\put(0,22){$\cdot$}

\put(28,44){\vector(1,-1){6}}
\put(20,50){$\cdot$}
\put(22,48){$\cdot$}
\put(24,46){$\cdot$}

\put(40,56){\vector(1,-1){6}}
\put(32,62){$\cdot$}
\put(34,60){$\cdot$}
\put(36,58){$\cdot$}

\put(52,68){\vector(1,-1){6}}
\put(44,74){$\cdot$}
\put(46,72){$\cdot$}
\put(48,70){$\cdot$}

\put(125,4){\vector(1,1){6}}
\put(133,12){$\cdot$}
\put(135,14){$\cdot$}
\put(137,16){$\cdot$}

\put(113,16){\vector(1,1){6}}
\put(121,24){$\cdot$}
\put(123,26){$\cdot$}
\put(125,28){$\cdot$}

\put(89,40){\vector(1,1){6}}
\put(97,48){$\cdot$}
\put(99,50){$\cdot$}
\put(101,52){$\cdot$}

\put(77,52){\vector(1,1){6}}
\put(85,60){$\cdot$}
\put(87,62){$\cdot$}
\put(89,64){$\cdot$}

\put(65,64){\vector(1,1){6}}
\put(73,72){$\cdot$}
\put(75,74){$\cdot$}
\put(77,76){$\cdot$}

\put(53,0){$\cdot$}
\put(55,0){$\cdot$}
\put(57,0){$\cdot$}

\end{picture}
}
}
\end{centering}
\vspace*{1mm}
\hfill \break
\item[(3)] 
$\Ext^{2}_{A}(X, Y ) = 0$, for each pair of objects $X$ and 
$Y$ of $\mathcal{E}$.
\end{enumerate}
\end{theorem}

\bigskip

\begin{theorem} \cite[Theorem 2.6 in Ch. X]{SS}
\label{xxthm2.5}
Retain the hypothesis as in Theorem \ref{xxthm2.4}.
Then the abelian category $\mathcal{E}$ has the following 
properties.
\begin{enumerate}
\item[(1)] 
Every indecomposable object $M$ of the category $\mathcal{E}$ 
is uniserial and is of the form $M\cong E_{i}[j]$, where 
$i\in \{1, \cdots, r\}$ and $j\geq 1$.
\item[(2)] 
The collection of indecomposable objects forms a self-hereditary 
component, denoted by $\mathfrak{T}_{\mathcal{E}}$, of 
$\Gamma(\Mod_{f.d.}-A)$.
\item[(3)] 
The component $\mathfrak{T}_{\mathcal{E}}$ is a standard stable 
tube of rank $r$ \cite[Definition 1.1 in Ch. X]{SS}.
\item[(4)] 
The objects $E_{1}, \cdots ,E_{r}$ form the complete set of 
objects lying on the mouth \cite[Definition 1.2 in Ch. X]{SS} 
of the tube $\mathfrak{T}_{\mathcal{E}}$.
\end{enumerate}
\end{theorem}

Let ${\mathfrak{T}}$ be the $\mathfrak{T}_{\mathcal{E}}$ 
defined as in Theorem \ref{xxthm2.5}. Let 
$D:=\Hom_{\Bbbk}(-,k)$ be the usual $\Bbbk$-linear dual.

\begin{corollary}\cite[Corollary 2.7 in Ch. X]{SS}
\label{xxcor2.6}
Retain the hypothesis as in Theorem \ref{xxthm2.4}.
\begin{enumerate}
\item[(1)] 
The only homomorphisms between two indecomposable
modules in $\mathfrak{T}$ are $\Bbbk$-linear combinations 
of compositions of the homomorphisms $u_{ij} , p_{ij} $, and 
the identity homomorphisms, and they are only subject to the 
relations arising from the almost split sequences in 
Theorem \ref{xxthm2.4}(1).
\item[(2)] 
Given $i \in \{1, \cdots , r\} $ and $j\geq 1$, we have
\begin{enumerate}
\item[(2a)]
$\End_A (E_{i}[j])\cong \Bbbk[t]/(t^{m})$, for some $m\geq 1$,
\item[(2b)]
$\End_A (E_{i}[j])\cong \Bbbk $ if and only if $j\leq r$, and
\item[(2c)]
$\Ext^{1}_{A}(E_{i}[j],E_{i}[j])\cong D\Hom_{A}(E_{i}[j], 
\tau E_{i}[j]) = 0$ if and only if $j\leq r-1$.
\end{enumerate}
\item[(3)] 
If the tube $\mathfrak{T}$ is homogeneous {\rm{(}}namely, 
$r=1${\rm{)}}, then $\Ext^{1}_{A}(M,M)\neq 0$, for any 
indecomposable $M$ in $\mathcal{C}$.
\end{enumerate}
\end{corollary}

Brick objects in ${\mathfrak{T}}$ are determined by Corollary 
\ref{xxcor2.6}(2b). To work out all brick sets in 
${\mathfrak{T}}$, we need to understand the $\Hom$ between 
brick objects. Part (2) of the following theorem describes 
these $\Hom$s. 

\begin{theorem}
\label{xxthm2.7}
Let $\mathfrak{T}$ be a standard stable tube of rank $r$ as used 
in Theorem \ref{xxthm2.5} and Corollary {\rm{\ref{xxcor2.6}}}. 
Keep the notation as above and assume $1\leq i,j, i',j'\leq r$. 
Then the following hold.
\begin{enumerate}
\item[(1)] 
$\End_{\mathfrak{T}}(E_{i}[j])\cong \Bbbk$.
\item[(2)] 
$\Hom_{\mathfrak{T}}(E_{i}[j], E_{i'}[j'])\neq 0$ if and only if 
$(i',j')$ satisfies one of the following conditions:
\begin{enumerate}
\item[(2a)]
$i\leq i'\leq i+j-1$ and $i+j\leq i'+j'$,
\item[(2b)] 
$i'\leq i+j-1-r$ and $i+j\leq i'+j'+r$. Here, if $i+j-1-r< 1$, 
then $\{1\leq i'\leq i+j-1-r\}=\emptyset.$
\end{enumerate}
Moreover, if $\Hom_{\mathfrak{T}}(E_{i}[j], E_{i'}[j'])\neq 0$, then 
$\Hom_{\mathfrak{T}}(E_{i}[j], E_{i'}[j'])\cong \Bbbk.$
\end{enumerate}
\end{theorem}

\begin{proof}
(1) See Corollary \ref{xxcor2.6}(2).

(2) Since  $\mathfrak{T}$ is a standard stable tube, it is a mesh category
\cite[Definition 2.4 in Ch. X]{SS}. 
Moreover, by Corollary \ref{xxcor2.6}(a), the only homomorphisms between 
two indecomposable modules in $\mathfrak{T}$ are $\Bbbk$-linear combinations 
of compositions of the homomorphisms $u_{ij} , p_{ij} $, and the identity 
homomorphisms, which subject to the relations arising from the almost split 
sequences in Theorem \ref{xxthm2.4}(a). By mesh relationship 
\cite[Definition 2.4 in Ch. X]{SS}, we 
obtain the description (2a) and (2b) for all objects $E_{i'}[j']$ that satisfy 
$\Hom_{\mathfrak{T}}(E_{i}[j], E_{i'}[j'])\neq 0$.

Moreover, if $j'< j$, $\Hom_{\mathfrak{T}}(E_{i}[j], E_{i'}[j'])$ is generated 
by composition morphisms 
{\small $$u_{i',j'}\cdots u_{i',i+j-i'+1+l} \cdots u_{i',i+j-i'+2}u_{i',i+j-i'+1} 
p_{i'-1,i+j-i'+1} \cdots p_{i+k,j-k}\cdots p_{i+1,j-1}p_{ij},$$} 
\noindent
where $0\leq k\leq i'-i, 0\leq l\leq (i'+j')-(i+j)-1$. Here, if $i+k> r$, 
then the index $i+k$ means $i+k-r$.

If $j'\geq j$, $\Hom_{\mathfrak{T}}(E_{i}[j], E_{i'}[j'])$ is generated by
{\small $$p_{i'-1,j'+1} \cdots p_{i+l,i'+j'-i-l} \cdots p_{i+1,i'+j'-i-1} 
p_{i,i'+j'-i}u_{i,i'+j'-i}\cdots u_{i,j+k} \cdots u_{i,j+2}u_{i,j+1} , $$}
\noindent
where $0\leq k\leq (i'+j')-(i+j), 0 \leq l\leq i'-i-1.$ Here, if $i+l> r$, 
then the index $i+l$ means $i+l-r$. Therefore 
$\Hom_{\mathfrak{T}}(E_{i}[j], E_{i'}[j'])\cong \Bbbk$. 
\end{proof}

Part (1) of Corollary \ref{xxcor2.8} next is just a re-interpretation of
Theorem \ref{xxthm2.7}(2). 

\begin{corollary}
\label{xxcor2.8}
Let $\mathfrak{T}$ be a standard stable tube of rank $r$. Keep the notation 
as above and put $E_{i}[j], 1\leq i,j\leq r$ in order
{\small
$$E_{1}[1], \; E_{2}[1], \; \cdots, \; E_{r}[1]; \; E_{1}[2], 
\; E_{2}[2], \; \cdots, \; E_{r}[2]; \; 
\cdots; \; E_{1}[r], \; E_{2}[r], \; \cdots, \; E_{r}[r],$$
}
and denote them by $X_1,\cdots, X_n$ where $n=r^{2}$. 
\begin{enumerate}
\item[(1)]
The $n\times n$ matrix 
$$\left( \dim \Hom_{\mathfrak{T}}(X_j, X_i) \right) _{n\times n}$$
has the following form
\begin{equation}
\label{E2.8.1}\tag{E2.8.1}
\begin{pmatrix}
P^{0} & P^{1}&P^{2}&P^{3}&\cdots& P^{r-1}\\
P^{0} & \sum\limits_{i=0}^{1}P^{i}&\sum\limits_{i=1}^{2}P^{i}
&\sum\limits_{i=2}^{3}P^{i}&\cdots& \sum\limits_{i=r-2}^{r-1}P^{i}\\
P^{0} & \sum\limits_{i=0}^{1}P^{i}&\sum\limits_{i=0}^{2}P^{i}
&\sum\limits_{i=1}^{3}P^{i}&\cdots& \sum\limits_{i=r-3}^{r-1}P^{i}\\
P^{0} & \sum\limits_{i=0}^{1}P^{i}&\sum\limits_{i=0}^{2}P^{i}
&\sum\limits_{i=0}^{3}P^{i}&\cdots& \sum\limits_{i=r-4}^{r-1}P^{i}\\
&\cdots&\cdots&\cdots&\cdots&\\
P^{0} & \sum\limits_{i=0}^{1}P^{i}&\sum\limits_{i=0}^{2}P^{i}
&\sum\limits_{i=0}^{3}P^{i}&\cdots& \sum\limits_{i=0}^{r-1}P^{i}\\
\end{pmatrix}
\end{equation}
where 
$$P=
\begin{pmatrix}
0& 0&0&0&0& 1\\
1& 0&0&0&0& 0\\
0& 1&0&0&0& 0\\
0& 0&1&0&0& 0\\
&\cdots&\cdots&\cdots&\cdots&\\
0& 0&0&0&1&0\\
\end{pmatrix}
_{r\times r}
$$
and where $P^{0}$ is the identity matrix $I_{r\times r}$ of order $r$.
\item[(2)]
The $n\times n$-matrix
$$\left( \dim \Ext^1_{\mathfrak{T}}(X_j, X_i) \right)_{n\times n}
$$
has the following form
\begin{equation}
\label{E2.8.2}\tag{E2.8.2}
\begin{pmatrix}
P^{r-1} & P^{r-1}&P^{r-1}&P^{r-1}&\cdots& P^{r-1}\\
P^{r-2} & \sum\limits_{i=r-2}^{r-1}P^{i}
&\sum\limits_{i=r-2}^{r-1}P^{i}&\sum\limits_{i=r-2}^{r-1}P^{i}
&\cdots& \sum\limits_{i=r-2}^{r-1}P^{i}\\
P^{r-3} & \sum\limits_{i=r-3}^{r-2}P^{i}
&\sum\limits_{i=r-3}^{r-1}P^{i}&\sum\limits_{i=r-3}^{r-1}P^{i}
&\cdots& \sum\limits_{i=r-3}^{r-1}P^{i}\\
P^{r-4} & \sum\limits_{i=r-4}^{r-3}P^{i}
&\sum\limits_{i=r-4}^{r-2}P^{i}&\sum\limits_{i=r-4}^{r-1}P^{i}
&\cdots& \sum\limits_{i=r-4}^{r-1}P^{i}\\
&\cdots&\cdots&\cdots&\cdots&\\
P^{0} & \sum\limits_{i=0}^{1}P^{i}
&\sum\limits_{i=0}^{2}P^{i}&\sum\limits_{i=0}^{3}P^{i}
&\cdots& \sum\limits_{i=0}^{r-1}P^{i}\\
\end{pmatrix}
.
\end{equation}
\end{enumerate}
\end{corollary}

\begin{proof}
(1) This follows from Theorem \ref{xxthm2.7}(2).

(2) The assertion follows from part (1) and the Serre duality
$$\Ext^{1}_{\mathfrak{T}}(E_{i}[j],E_{i'}[j'])\cong 
D\Hom_{\mathfrak{T}}(E_{i'}[j'], 
\tau E_{i}[j])=D\Hom_{\mathfrak{T}}(E_{i'}[j'], 
E_{i-1}[j]).$$
Some detailed matching of entries is omitted.
\end{proof}

We use the following example to illustrate the results in Corollary 
\ref{xxcor2.8}.

\begin{example}
\label{xxex2.9}
Let $\mathfrak{T}$ be a standard stable tube of rank $3$: 

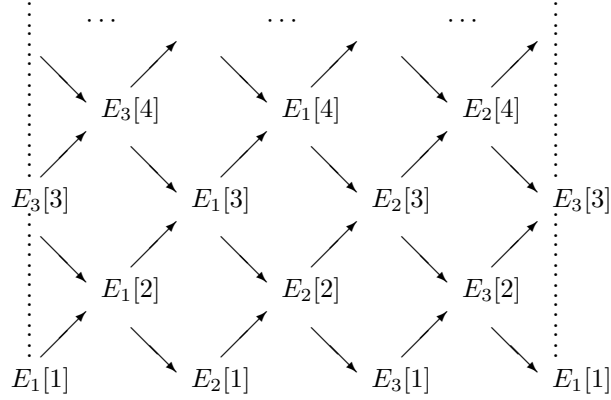
\unitlength=1mm
\begin{picture}(140,60)(-15,0)
\put(4,4){\vector(1,1){6}}
\put(28,4){\vector(1,1){6}}
\put(52,4){\vector(1,1){6}}
\put(16,8){\vector(1,-1){6}}
\put(40,8){\vector(1,-1){6}}
\put(64,8){\vector(1,-1){6}}

\put(4,20){\vector(1,-1){6}}
\put(16,16){\vector(1,1){6}}
\put(28,20){\vector(1,-1){6}}
\put(40,16){\vector(1,1){6}}
\put(52,20){\vector(1,-1){6}}
\put(64,16){\vector(1,1){6}}

\put(4,28){\vector(1,1){6}}
\put(28,28){\vector(1,1){6}}
\put(52,28){\vector(1,1){6}}
\put(16,32){\vector(1,-1){6}}
\put(40,32){\vector(1,-1){6}}
\put(64,32){\vector(1,-1){6}}

\put(4,44){\vector(1,-1){6}}
\put(16,40){\vector(1,1){6}}
\put(28,44){\vector(1,-1){6}}
\put(40,40){\vector(1,1){6}}
\put(52,44){\vector(1,-1){6}}
\put(64,40){\vector(1,1){6}}

\put(0,0){$E_{1}[1]$}
\put(24,0){$E_{2}[1]$}
\put(48,0){$E_{3}[1]$}
\put(72,0){$E_{1}[1]$}
\put(12,12){$E_{1}[2]$}
\put(36,12){$E_{2}[2]$}
\put(60,12){$E_{3}[2]$}
\put(0,24){$E_{3}[3]$}
\put(24,24){$E_{1}[3]$}
\put(48,24){$E_{2}[3]$}
\put(72,24){$E_{3}[3]$}
\put(12,36){$E_{3}[4]$}
\put(36,36){$E_{1}[4]$}
\put(60,36){$E_{2}[4]$}

\put(2,4){$\vdots$}
\put(2,8){$\vdots$}
\put(2,12){$\vdots$}
\put(2,16){$\vdots$}
\put(2,20){$\vdots$}
\put(2,28){$\vdots$}
\put(2,32){$\vdots$}
\put(2,36){$\vdots$}
\put(2,40){$\vdots$}
\put(2,44){$\vdots$}
\put(2,48){$\vdots$}

\put(72,4){$\vdots$}
\put(72,8){$\vdots$}
\put(72,12){$\vdots$}
\put(72,16){$\vdots$}
\put(72,20){$\vdots$}
\put(72,28){$\vdots$}
\put(72,32){$\vdots$}
\put(72,36){$\vdots$}
\put(72,40){$\vdots$}
\put(72,44){$\vdots$}
\put(72,48){$\vdots$}

\put(10,48){$\cdots$}
\put(34,48){$\cdots$}
\put(58,48){$\cdots$}
\end{picture}
\vspace*{5mm}
\hfill \break

Put $E_{i}[j], 1\leq i,j\leq 3$ in order 
$$E_{1}[1],E_{2}[1], E_{3}[1]; E_{1}[2], E_{2}[2], 
E_{3}[2]; E_{1}[3], E_{2}[3], E_{3}[3].$$ 
Denote this list as $X_1,\cdots, X_9$. Then we have

{\tiny 
\begin{table}[h]
\caption{$\Hom_{\mathfrak{T}}(X_j,X_i)$}
\begin{tabular}{|c|c|c|c|c|c|c|c|c|c|}
  \hline
  $\Hom_{\mathfrak{T}}(X_j,X_i)$ & 
	$X_j=E_{1}[1]$& $E_{2}[1]$& $E_{3}[1]$& 
	$E_{1}[2]$& $E_{2}[2]$& $E_{3}[2]$& 
	$E_{1}[3]$& $E_{2}[3]$&$ E_{3}[3]$ \\
  \hline
$ X_i=E_{1}[1]$&$\Bbbk$&0&0&0&0&$\Bbbk$&0&$\Bbbk$&0\\
  \hline
  $E_{2}[1]$&0&$\Bbbk$&0&$\Bbbk$&0&0&0&0&$\Bbbk$  \\
  \hline
  $ E_{3}[1]$&0&0&$\Bbbk$&0&$\Bbbk$&0&$\Bbbk$&0&0  \\
  \hline
    $E_{1}[2]$&$\Bbbk$&0&0&$\Bbbk$&0&$\Bbbk$&0&$\Bbbk$&$\Bbbk$ \\
  \hline
  $E_{2}[2]$&0&$\Bbbk$&0&$\Bbbk$&$\Bbbk$&0&$\Bbbk$&0&$\Bbbk$\\
  \hline
  $E_{3}[2]$&0&0&$\Bbbk$&0&$\Bbbk$&$\Bbbk$&$\Bbbk$&$\Bbbk$&0\\
  \hline
  $E_{1}[3]$&$\Bbbk$&0&0&$\Bbbk$&0&$\Bbbk$&$\Bbbk$&$\Bbbk$&$\Bbbk$\\ 
   \hline
  $E_{2}[3]$&0&$\Bbbk$&0&$\Bbbk$&$\Bbbk$&0&$\Bbbk$&$\Bbbk$&$\Bbbk$\\
   \hline
  $E_{3}[3]$&0&0&$\Bbbk$&0&$\Bbbk$&$\Bbbk$&$\Bbbk$&$\Bbbk$&$\Bbbk$\\
  \hline
\end{tabular}
\end{table}
}

{\tiny 
\begin{table}[h]
\caption{$\Ext^{1}_{\mathfrak{T}}(X_j,X_i)$}
\begin{tabular}{|c|c|c|c|c|c|c|c|c|c|}
  \hline
  $\Ext^{1}_{\mathfrak{T}}(X_j,X_i)$ & 
	$X_j=E_{1}[1]$& $E_{2}[1]$& $E_{3}[1]$
	& $E_{1}[2]$& $E_{2}[2]$& $E_{3}[2]$
	& $E_{1}[3]$& $E_{2}[3]$&$ E_{3}[3]$ \\
  \hline
$ X_i=E_{1}[1]$&0&$\Bbbk$&0&0&$\Bbbk$&0&0&$\Bbbk$&0\\
  \hline
  $E_{2}[1]$&0&0&$\Bbbk$&0&0&$\Bbbk$&0&0&$\Bbbk$  \\
  \hline
  $ E_{3}[1]$&$\Bbbk$&0&0&$\Bbbk$&0&0&$\Bbbk$&0&0  \\
  \hline
    $E_{1}[2]$&0&0&$\Bbbk$&0&$\Bbbk$&$\Bbbk$&0&$\Bbbk$&$\Bbbk$ \\
  \hline
  $E_{2}[2]$&$\Bbbk$&0&0&$\Bbbk$&0&$\Bbbk$&$\Bbbk$&0&$\Bbbk$\\
  \hline
  $E_{3}[2]$&0&$\Bbbk$&0&$\Bbbk$&$\Bbbk$&0&$\Bbbk$&$\Bbbk$&0\\
  \hline
  $E_{1}[3]$&$\Bbbk$&0&0&$\Bbbk$&0&$\Bbbk$&$\Bbbk$&$\Bbbk$&$\Bbbk$\\
   \hline
  $E_{2}[3]$&0&$\Bbbk$&0&$\Bbbk$&$\Bbbk$&0&$\Bbbk$&$\Bbbk$&$\Bbbk$\\
   \hline
  $E_{3}[3]$&0&0&$\Bbbk$&0&$\Bbbk$&$\Bbbk$&$\Bbbk$&$\Bbbk$&$\Bbbk$\\
  \hline
\end{tabular}
\end{table}
}

The corresponding $\Hom$-dimension and $\Ext^1$-dimension matrices are 
$$\left( \dim \Hom_{\mathfrak{T}}(X_j,X_i) \right)_{9\times 9}
=
\begin{pmatrix}
P^{0} & P^{1}&P^{2}\\
P^{0} & \sum\limits_{i=0}^{1}P^{i}&\sum\limits_{i=1}^{2}P^{i}\\
P^{0} & \sum\limits_{i=0}^{1}P^{i}&\sum\limits_{i=0}^{2}P^{i}
\end{pmatrix}
 $$
and 
$$
\left(\dim \Ext^1_{\mathfrak{T}}(X_j,X_i) \right)_{9\times 9} 
=
\begin{pmatrix}
P^{2} & P^{2}&P^{2}\\
P^{1} & \sum\limits_{i=1}^{2}P^{i}&\sum\limits_{i=1}^{2}P^{i}\\
P^{0} & \sum\limits_{i=0}^{1}P^{i}&\sum\limits_{i=0}^{2}P^{i}
\end{pmatrix}
$$
where 
$$P=
\begin{pmatrix}
0 & 0&1\\
1 & 0&0\\
0 & 1&0
\end{pmatrix}.$$
\end{example}

The next lemma shows that if $H(i_1,\cdots,i_s)$ is a principal 
submatrix of \eqref{E2.8.1} such that $H(i_1,\cdots, i_s)=I_{s\times s}$,
then the corresponding principal 
submatrix of \eqref{E2.8.2} (with the same rows and columns) 
has spectral radius at most 1. 

\begin{lemma}[Lemma \ref{xxlem6.4}]
\label{xxlem2.10}
Retain the above notation.
Suppose the $\Hom$-matrix is given as in \eqref{E2.8.1} and
$\Ext^1$-matrix is given as in \eqref{E2.8.2}, then 
$\rho(A(\phi))\leq 1$ for every brick set $\phi$.
\end{lemma}

The proof of Lemma \ref{xxlem2.10} is given in Appendix A.
Note that \eqref{E2.8.1}-\eqref{E2.8.2} are in fact the 
transpose of the usual Hom and Ext-matrices. By \cite[Lemma 3.7]{CG}
(by considering the opposite category), Lemma \ref{xxlem2.10}
holds for the transpose matrices of \eqref{E2.8.1}-\eqref{E2.8.2} 
too.

\begin{theorem}
\label{xxthm2.11}
Let $\mathcal{E}$ and $\mathfrak{T}$ be as in Theorem \ref{xxthm2.5}.
Then $\fpdim {\mathcal{E}}=1$.
\end{theorem}

\begin{proof} By Corollary \ref{xxcor2.6}(2b), all brick 
objects in ${\mathcal{E}}$ are $E_i[j]$ for all $i\in {\mathbb Z}/(r)$
and $1\leq j\leq r$. We can determine all brick sets by using 
the matrix in Corollary \ref{xxcor2.8}(1). For each brick set, its
$\Ext^1$-matrix $M$ was determined by using \eqref{E2.8.2}. By 
Lemma  \ref{xxlem2.10}, $\rho(M)\leq 1$. On the other hand, 
let $\phi=\{E_1[1],\cdots, E_r[1]\}$, then $M=P^{r-1}$ and hence
$\fpdim \mathcal{E}=\rho(A(\phi, \Ext^1))=1$. 
Therefore the assertion follows.
\end{proof}

We have an immediate consequence. Let $T_r$ be the algebra 
in Example \ref{xxex2.3}.

\begin{corollary}
\label{xxcor2.12} 
Let ${\mathfrak A}$ be the category of finite 
dimensional left $x$-nilpotent $T_r$-modules.
Then $\fpdim {\mathfrak A}=1$.
\end{corollary}

\begin{proof} It is well-known that ${\mathfrak A}$ is equivalent 
to the category ${\mathcal{E}}$ of rank $r$. 
To see this we set $\deg x=1$ and $\deg g=0$. Then the degree zero
component of $T_r$ is isomorphic to 
$\Bbbk^{\oplus r}$ with primitive idempotents $\{e_1,
\cdots, e_r\}$. Under this setting, $E_i[j]$ is identify
with $(T_r/(x^i))e_j$ for all $i,j$. Now the result follows from 
Theorem \ref{xxthm2.11}.
\end{proof}

\subsection{Proof of Theorem \ref{xxthm0.1}}
\label{xxsec2.3}
Now we are ready to show Theorem \ref{xxthm0.1}.

\begin{theorem}
\label{xxthm2.13} Let ${\mathbb X}$ be a domestic or tubular weighted 
projective line. Then $\fpdim\; D^b(coh({\mathbb X}))=1$.
\end{theorem}

\begin{proof} By Lemma \ref{xxlem2.1}(4), it suffices to show that 
$\fpdim (D^b(coh({\mathbb X})))\leq 1$. By \cite[Theorem 3.5(4)]{CG},
it is enough to show that 
$$\fpdim (\sigma)=\fpdim (coh({\mathbb X}))\leq 1$$ 
where $\sigma$ is $\Ext^1_{\mathbb X}(-,-)$. 

By \cite[Corollary 4.34(iii)]{Sc}, every brick (or indecomposable) 
object is semistable. By Lemma \ref{xxlem2.1}(5), the class 
$\{ Vect_{\mu}({\mathbb X})\}_{\mu \in {\mathbb Q}\cup\{\infty\}}$
is a $\sigma$-decomposition (see the definition before Lemma \ref{xxlem1.6}). 
By Lemma \ref{xxlem1.6}, it is enough to show the claim that 
$\fpdim \mid_{Vect_{\mu}({\mathbb X})} (\sigma) \leq 1$ for every $\mu$. 

Case 1: ${\mathbb X}$ is domestic. If $\mu$ is finite, then 
$\fpdim\mid_{Vect_{\mu}({\mathbb X})} (\sigma)=0$ by Lemma \ref{xxlem2.1}(6). 
If $\mu=\infty$, then, by \eqref{E2.0.1}, $Vect_{\infty}({\mathbb X}):=
Tor({\mathbb X})$ has a decomposition into Auslander-Reiten 
components, which all are tubes of finite rank. By Theorem \ref{xxthm2.11},
$\fpdim\mid_{Tor({\mathbb X})}(\sigma)=1$. The claim 
follows.

Case 2: ${\mathbb X}$ is tubular. By Lemma \ref{xxlem2.1}(8),
$$Vect_{\mu}({\mathbb X})\cong Vect_{\infty}({\mathbb X})
=Tor({\mathbb X})$$
for all $\mu$. Then the proof of Case 1 applies. Therefore 
the claim follows.
\end{proof}

As usual we use ${\mathcal O}_{\mathbb X}$ for the 
structure sheaf of ${\mathbb X}$ \cite[Sect. 1.5]{GL}.

\begin{proposition}
\label{xxpro2.14} Let ${\mathbb X}$ be a weighted projective line
of wild type. Then 
$$\fpdim\; D^b(coh({\mathbb X}))
\geq \dim 
\Hom_{\mathbb X}({\mathcal O}_{\mathbb X},
{\mathcal O}_{\mathbb X}(\overrightarrow{\omega}))$$ 
where $\overrightarrow{\omega}$ is the dualizing element \cite[Sec. 1.2]{GL}. 
\end{proposition}

\begin{proof} Let $\phi=\{{\mathcal O}_{\mathbb X}\}$
which is a brick and atomic object. Then, by definition 
and Serre duality \cite[Theorem 2.2]{GL},
$$\begin{aligned}
\fpdim\; D^b(coh({\mathbb X})) &\geq \rho(A(\phi, \Ext^1))=
\dim \Ext^1_{\mathbb X}
({\mathcal O}_{\mathbb X},{\mathcal O}_{\mathbb X})\\
&=\dim \Hom_{\mathbb X}({\mathcal O}_{\mathbb X},
{\mathcal O}_{\mathbb X}(\overrightarrow{\omega})).
\end{aligned}
$$
\end{proof}

\section{Dimension theory for classical
projective schemes}
\label{xxsec3}

The aim of this section is to introduce Frobenius-Perron versions
of two important and related invariants -- Calabi-Yau 
dimension and Kodaira dimension of $D^b(coh({\mathbb X}))$
where ${\mathbb X}$ is a smooth projective scheme.

\subsection{A result from \cite{CG}}
\label{xxsce3.1}

Let ${\mathbb X}$ be a smooth (irreducible) projective scheme 
over ${\mathbb C}$ of positive dimension. By 
\cite[Proposition 6.5 and 6.7]{CG},
$$\fpdim (D^b(coh({\mathbb X})))=\begin{cases} 1& 
{\text{if ${\mathbb X}$ is ${\mathbb P}^1$ or an elliptic curve}},\\
\infty & {\text{otherwise.}}
\end{cases}$$

\subsection{Calabi-Yau dimension}
\label{xxsec3.2}
Recall from \cite[Section 8.1]{Ke1} that if a $\Hom$-finite category 
${\mathcal C}$ has a Serre functor $S$, then there is a natural 
isomorphism 
$$\Hom_{\mathcal C}(X,Y)^*\cong \Hom_{\mathcal C}(Y, S(X))$$
for all $X,Y\in {\mathcal C}$. A (pre-)triangulated $\Hom$-finite 
category ${\mathcal C}$ with Serre functor $S$ is called 
{\it $n$-Calabi-Yau} if there is a natural isomorphism
$$S\cong \Sigma^n =:[n]$$
where $\Sigma$ is the suspension of ${\mathcal C}$. In this case 
$n$ is called the Calabi-Yau dimension of ${\mathcal C}$.
(In \cite[Section 2.6]{Ke2} it is called {\it weakly $n$-Calabi-Yau}.)
More generally, ${\mathcal C}$ is called a {\it fractional Calabi-Yau 
category} if there is an $m>0$ and there is a natural isomorphism
$$S^{m}\cong \Sigma^n=[n]$$
for some $n$, see \cite[p.2708]{vR} and \cite[Definition 1.2]{Ku}. 
In this case we say ${\mathcal C}$ has Calabi-Yau dimension $\frac{n}{m}$. 
Abelian hereditary fractionally Calabi-Yau categories are classified in 
\cite{vR}. One key property of Calabi-Yau varieties is that the 
canonical bundle of these varieties are trivial. However, our definition 
of fp Calabi-Yau dimension (see Definition \ref{xxdef3.1} below) applies 
to projective schemes that do not have the trivial canonical bundle.

If a Serre functor exists, then it is unique up to isomorphism.
In this case we usually use $S$ for the Serre functor. 
Throughout the rest of this section, let ${\mathcal T}$ be a 
$\Hom$-finite (pre-)triangulated category with Serre functor $S$. We will 
define a version of (fractional) Calabi-Yau dimension for ${\mathcal T}$ 
which is not necessarily a (fractional) Calabi-Yau category. 

Recall from Definition \ref{xxdef1.4}(3) that the Frobenius-Perron 
growth of a functor $\sigma$ is defined to be 
$$\fpg(\sigma):=\sup_{\phi\in \Phi_{a}} 
\{\limsup_{n\to\infty} \; \log_{n}(\rho(A(\phi,\sigma^n))) \}.$$
By convention, $\log_n 0 =-\infty$. Similarly, we define
a slightly modified version of $\fpg$ as follows, which is used
to define the fp Calabi-Yau dimension.

\begin{definition}
\label{xxdef3.1} 
Let $\sigma$ be an endofunctor of ${\mathcal T}$ with Serre functor
$S$. 
\begin{enumerate}
\item[(1)]
The {\it lower Frobenius-Perron growth} of $\sigma$ is defined to be
$$\underline{\fpg}(\sigma):=\sup_{\phi\in \Phi_{a}} 
\left\{\liminf_{n\to\infty} \; \log_{n}(\rho(A(\phi,\sigma^n))) \right\}.$$
By convention, $\log_n 0 =-\infty$.
\item[(2)]
The {\it spectrum} of ${\mathcal T}$ is defined to be
$$Sp({\mathcal T}):=\left\{(m,n)\in {\mathbb Z}^{\times 2}\mid
\underline{\fpg}(S^{m}\circ \Sigma^{-n})>-\infty\right\}.$$
\item[(3)]
The {\it fp Calabi-Yau dimension} of ${\mathcal T}$ is defined to be
$$\fpcy({\mathcal T}):=\lim_{M\to\infty}
\left\{\sup_{|m|\geq M}\left\{\frac{n}{m} \mid (m,n)\in Sp({\mathcal T})\right\}
\right\}.$$
\end{enumerate}
\end{definition}

Next we show that the fp Calabi-Yau dimension exists for various cases.

\begin{lemma}
\label{xxlem3.2} Suppose ${\mathcal T}$ is a pre-triangulated category 
satisfying the following conditions:
\begin{enumerate}
\item[(a)]
${\mathcal T}$ is Ext-finite \cite[Definition 2.1]{BV}, namely, for all 
objects $X,Y\in {\mathcal T}$, 
$$\sum_{s\in {\mathbb Z}}
\dim \Hom_{\mathcal T}(X, \Sigma^s(Y))<\infty.$$
\item[(b)]
${\mathcal T}$ has a Serre functor $S$.
\item[(c)]
${\mathcal T}$ is fractional Calabi-Yau of dimension $d=a/b\in {\mathbb Q}$.
\end{enumerate}
Then the following holds.
\begin{enumerate}
\item[(1)]
$Sp({\mathcal T})\subseteq (b,a){\mathbb Q}$.
\item[(2)]
If ${\mathcal T}$ contains at least one atomic object,  
there exists $w\in {\mathbb N}$ such that 
$(bwt, awt)\in Sp({\mathcal T})$ for all 
$t\in {\mathbb Z}$.
\item[(3)]
Under the hypothesis of part (2), we have 
$\fpcy({\mathcal T})=d$.
\end{enumerate}
\end{lemma}

\begin{proof} (1) Let $(m,n)$ be a pair of integers that is not 
in $(b,a){\mathbb Q}$, and let $\sigma=S^{m}\circ \Sigma^{-n}$. 
Since ${\mathcal T}$ is fractional Calabi-Yau of dimension $a/b$, 
we can assume that $S^b\circ \Sigma^{-a}$ is the identity functor. 
Then, for each $t$, $\sigma^{tb}=\Sigma^{t(ma-nb)}$ where 
$ma-nb\neq 0$. By hypothesis (a), for any two objects $X$ and 
$Y$ in ${\mathcal T}$, $\Hom_{\mathcal T}(X,\Sigma^{t(ma-nb)} Y)=0$ 
for $|t|\gg 0$. Then, for every atomic set $\phi$, it implies that 
$A(\phi, \sigma^{tb})=0$ for $|t|\gg 0$. Thus 
$\log_{tb}(\rho(A(\phi,\sigma^{tb})))=-\infty$ for all $|t|\gg 0$. 
This implies, by definition, that $\underline{\fpg}(\sigma)=
-\infty$, or $(m,n)\not\in Sp({\mathcal T})$. The assertion follows.

(2) Let $\phi$ be the set of a single atomic object $X$ in 
${\mathcal T}$. Replacing $(b,a)$ by $(bw, aw)$ for some 
positive integer $w$ if necessary, we can assume that 
$\sigma:=S^b\circ \Sigma^{-a}$ is the identity functor. Then 
$A(\phi, \sigma^n)$ is the $1\times 1$-identity matrix $I_1$ for all $n$. 
Then $\log_{n}(\rho(A(\phi,\sigma^{n})))=0$ for all $n$. This 
implies that $\underline{\fpg}(\sigma)=0>-\infty$ or 
$(b,a)\in S({\mathcal T})$. Similarly, one sees that 
$(bt, at)\in S({\mathcal T})$ for all integer $t$.
The assertion follows.

(3) This follows from the definition and parts (1,2).
\end{proof}

The next proposition is part (2) of Theorem \ref{xxthm0.2}, 
which shows that $\fpcy$ is indeed a generalization of 
Calabi-Yau dimension.

\begin{proposition}
\label{xxpro3.3}
Let ${\mathbb X}$ be a smooth irreducible projective variety 
of dimension $d\in {\mathbb N}$ and let ${\mathcal T}$ 
be $D^b(coh({\mathbb X}))$. Then $\fpcy({\mathcal T})=d$.
\end{proposition}

\begin{proof} When $d=0$, then ${\mathcal T}=D^b_{f.d}(\Vect_k)$.
It is easy to see that the Serre functor $S$ is the identity.
So the assertion is easily shown. Now we assume that $d>0$.

It suffices to show that $Sp({\mathcal T})=
\{(t,td)\mid t\in {\mathbb Z}\}$. 

By \cite[(7)]{BO}, the Serre functor $S$ is equal to 
$-\otimes_{{\mathbb X}} \omega_{\mathbb X}[d]$ where $d$ is 
the dimension of ${\mathbb X}$ and $\omega_{\mathbb X}$ is 
the canonical bundle of ${\mathbb X}$. Let 
$\sigma=S\circ \Sigma^{-d}$. Then $\sigma$ is the functor 
$-\otimes_{{\mathbb X}} \omega_{\mathbb X}$. Let ${\mathcal O}_a$ 
be the skyscraper sheaf of a closed point $a\in {\mathbb X}$. 
Then it is an atomic object and $\sigma({\mathcal O}_a)\cong 
{\mathcal O}_a$. Let $\phi=\{{\mathcal O}_a\}$. Then 
$A(\phi, \sigma^n)$ is the $1\times 1$-identity matrix $I_1$. 
This implies that $\log_{n}(\rho(A(\phi,\sigma^{n})))=0$ and 
that $\underline{\fpg}(\sigma)=0>-\infty$. Therefore $(1,d)\in
S({\mathcal T})$. Similarly, one sees that $(t,dt)\in
S({\mathcal T})$ for all $t\geq 0$. 

For the other implication, let $(m,n)\in {\mathbb Z}^{\times 2}
\setminus \{(t,td)\mid t\in {\mathbb Z}\}$. Let
$\sigma= S^m\circ \Sigma^{-n}$. We need to show that
that $\underline{\fpg}(\sigma)=-\infty$.
Note that $\sigma=-\otimes_{{\mathbb X}}
\omega_{{\mathbb X}}^{\otimes m} \circ \Sigma^{-n+dm}$ where 
$-n+dm\neq 0$. Since $coh({\mathbb X})$ has global dimension 
$d$, for all objects $A$ and $B$ in ${\mathcal T}$,
$$\Hom_{\mathcal T}(A, \sigma^t(B))=
\Hom_{\mathcal T}(A, (B\otimes_{{\mathbb X}}
\omega_{{\mathbb X}}^{tm}) [t(-n+dm)])=0$$
for all $t\gg 0$. This implies that 
$\underline{\fpg}(\sigma)=-\infty$ as required.
\end{proof}

\subsection{Kodaira dimension}
\label{xxsec3.3}
First we review the classical definition of the Kodaira 
dimension. 

\begin{definition} \cite[Definition 2.1.3 and Example 2.1.5]{La}
\label{xxdef3.4}
Let ${\mathbb X}$ be a smooth projective variety and let 
$\omega_{\mathbb X}$ be the canonical bundle of ${\mathbb X}$. 
\begin{enumerate}
\item[(1)]
The {\it Kodaira dimension} of ${\mathbb X}$ is defined to be 
$$\kappa ({\mathbb X}):=\lim_{n\to\infty}
\log_n \left(\dim H^0({\mathbb X}, 
\omega_{\mathbb X}^{\otimes n})\right).$$
\item[(2)]
More generally, for a line bundle ${\mathcal M}$, the 
{\it Kodaira-Iitaka dimension} of ${\mathcal M}$ is defined to be
$$\kappa({\mathbb X},{\mathcal M}):=\lim_{n\to\infty}
\log_n \left(\dim H^0({\mathbb X}, {\mathcal M}^{\otimes n})\right).$$
\item[(3)]
The {\it anti-Kodaira dimension} of ${\mathbb X}$ is defined to be
$$\kappa^{-1}(X):= \kappa({\mathbb X}, \omega_{\mathbb X}^{-1}).$$
\end{enumerate}
\end{definition}

The anti-Kodaira  dimension of a scheme was defined in 
\cite{Sa}. It is classical and well-known that 
$\kappa({\mathbb X}), \kappa({\mathbb X}, {\mathcal M})
\in \{-\infty, 0,1, \cdots, \dim {\mathbb X}\}$ and that there are 
$0<c_1<c_2$ such that,
\begin{equation}
\label{E3.4.1}\tag{E3.4.1}
c_1 n^{\kappa({\mathbb X}, {\mathcal M})} \leq \dim 
\Hom_{\mathbb X}({\mathcal O}_{\mathbb X}, {\mathcal M}^{\otimes n})
\leq c_2 n^{\kappa({\mathbb X}, {\mathcal M})} \quad \forall \; n\gg 0.
\end{equation}
See \cite[Corollary 2.1.37]{La}.
By Proposition \ref{xxpro3.3}, $\dim {\mathbb X}=\fpcy({\mathcal T})$, 
which suggests the following definition.

In the following we use the order $t_1<t_2$ if $t_1$ divides $t_2$.
Since we are mainly interested in commutative and noncommutative 
projective schemes, ${\mathcal T}$ is equipped with the 
structure sheaf, denoted by ${\mathcal O}$.

\begin{definition}
\label{xxdef3.5}
Let ${\mathcal T}$ be a pre-triangulated category and 
${\mathcal T}_{\ast}$ denote the pair $({\mathcal T}, {\mathcal O})$
where ${\mathcal O}$ is a given special object in ${\mathcal T}$. 
Suppose that ${\mathcal T}$ is Hom-finite with Serre functor $S$ 
and that $\fpcy({\mathcal T})=d=a/b$ for some integers $a,b$. 
\begin{enumerate}
\item[(1)]
The {\it fp Kodaira  dimension} of ${\mathcal T}_{\ast}$ 
is defined to be
$$\fpk({\mathcal T}_{\ast})
:=\lim_{t} \left\{ \limsup_{n\to \infty} 
\dim \Hom_{\mathcal T}({\mathcal O}, (S^{bnt}\circ \Sigma^{-ant})
({\mathcal O}))\right\}$$
where the first limit ranges over all positive integers
$t$ with order $<$ as defined before the definition.
\item[(2)]
The {\it fp anti-Kodaira  dimension} of 
${\mathcal T}_{\ast}$ is defined to be
$$\fpk^{-1}({\mathcal T}_{\ast})
:=\lim_{t} \left\{ \limsup_{n\to \infty} 
\dim \Hom_{\mathcal T}({\mathcal O}, (S^{-bnt}\circ \Sigma^{ant})
({\mathcal O}))\right\}$$
where the first limit ranges over all positive integers
$t$ with order $<$ as defined before the definition.
\end{enumerate}
\end{definition}

The following proposition justifies 
the above definition.

\begin{proposition}
\label{xxpro3.6}
Let ${\mathbb X}$ be a smooth irreducible projective variety of 
dimension $d\in {\mathbb N}$ and let ${\mathcal T}$ 
be $D^b(coh({\mathbb X}))$ with structure sheaf $\mathcal{O}:
=\mathcal{O}_{\mathbb X}$. Then
$$\fpk({\mathcal T}_{\ast})=\kappa({\mathbb X})
\quad {\text{and}}\quad
\fpk^{-1}({\mathcal T}_{\ast})=\kappa^{-1}({\mathbb X}).$$
\end{proposition}

\begin{proof} 
By Proposition \ref{xxpro3.3}, 
$\fpcy({\mathcal T})=d=\dim {\mathbb X}$. So we take 
$b=1$ and $a=d$ in Definition \ref{xxdef3.5}. By \eqref{E3.4.1}, 
for each $t\geq 1$, 
$$\limsup_{n\to \infty} 
\dim \Hom_{\mathcal T}({\mathcal O}, (S^{btn}\circ \Sigma^{-ant})
({\mathcal O}))=\kappa({\mathbb X}).$$ 
The first assertion follows by the definition.

The proof for anti-Kodaira dimension is similar.
%
\end{proof}

Theorem \ref{xxthm0.2} follows from Propositions
\ref{xxpro3.3} and \ref{xxpro3.6}.

\begin{remark}
\label{xxrem3.7} 
Assume the hypothesis of Lemma \ref{xxlem3.2}.
Then one can check easily that 
\begin{equation}
\label{E3.7.1}\tag{E3.7.1}
\fpk({\mathcal T}_{\ast})=\fpk^{-1}({\mathcal T}_{\ast})=0.
\end{equation}
So abstractly \eqref{E3.7.1} should be part of the definition 
of a fractional Calabi-Yau variety (even in the noncommutative
setting).
\end{remark}

\begin{proposition}
\label{xxpro3.8}
If ${\mathcal T}$ is a triangulated category such that
either $\fpk({\mathcal T}_{\ast})=\infty$ or 
$\fpk^{-1}({\mathcal T}_{\ast})=\infty$ or $\fpcy({\mathcal T})=
\infty$ or $-\infty$, then ${\mathcal T}_{\ast}$ is not 
triangulated equivalent to the bounded derived category
of a smooth projective scheme. 
\end{proposition}

\begin{proof} By definition, $\fpcy$ is an invariant of 
a triangulated category, and $\fpk$ is an invariant of a
triangulated category with ${\mathcal O}$. 
The assertion follows from Propositions \ref{xxpro3.3}
and \ref{xxpro3.6}.
\end{proof}

\section{Invariants of noncommutative 
projective schemes}
\label{xxsec4}

In this section we study fp Calabi-Yau dimension and 
fp Kodaira dimension of noncommutative projective schemes 
in the sense of \cite{AZ}. An algebra $A$ is said to be 
{\it connected graded} over $\Bbbk$ if 
$A=\Bbbk\oplus A_1\oplus A_2 \oplus \cdots$ with 
$A_i A_j\subseteq A_{i+j}$ for all $i,j,\in {\mathbb N}$. 
Let $A$ be a noetherian connected graded algebra. The 
{\it noncommutative projective scheme} associated to 
$A$ is denoted by ${\mathbb X}:=\Proj A$, see \cite{AZ} 
for the detailed definition of a noncommutative projective 
scheme. Let $coh({\mathbb X})$ be the category of noetherian 
objects in $\Proj A$ and let ${\mathcal T}$ be the triangulated 
category $D^b(coh({\mathbb X}))$. Here is a restatement of a 
nice result of Bondal-Van den Bergh \cite[Theorem 4.2.13]{BV}.

\begin{theorem}\cite[Theorem 4.2.13]{BV}
\label{xxthm4.1}
Suppose $A$ is noetherian and has a balanced dualizing complex
and that $\Proj A$ has finite homological dimension. Then 
${\mathcal T}$ has a Serre functor.
\end{theorem}

One special class of connected graded algebras are the 
Artin-Schelter regular algebras (see the next definition).

\begin{definition} \cite{AS}
\label{xxdef4.2}
Let $A$ be a connected graded algebra over the base field $\Bbbk$.
We say $A$ is {\it Artin-Schelter Gorenstein}
(or {\it AS Gorenstein}) if the following conditions hold:
\begin{enumerate}
\item[(a)]
$A$ has finite injective dimension $d$ on both sides,
\item[(b)]
$\Ext^i_A(\Bbbk, A) = \Ext^i_{A^{op}} (\Bbbk, A) = 0$ for all
$i\neq d$ where $\Bbbk = A/A_{\geq 1}$, and
\item[(c)]
$\Ext^d_A(\Bbbk, A)\cong \Bbbk(\ell)$
and $\Ext^d_{A^{op}} (\Bbbk, A)\cong
\Bbbk(\ell)$ for some integer $\ell$.
This integer $\ell$ is called the {\it AS index} of $A$.
\end{enumerate}
If moreover
\begin{enumerate}
\item[(d)]
$A$ has finite global dimension, 
\end{enumerate}
then $A$ is called {\it Artin-Schelter regular} (or {\it AS regular}).
\end{definition}

We collect some well-known facts below. Let
$\pi: \Gr A\to \Proj A$ be the canonical 
quotient functor. By abuse of notation, 
we also apply $\pi$ to some graded $A$-bimodules
such as ${^{\mu} A^1}$ in Lemma \ref{xxlem4.3} 
below.

\begin{lemma}
\label{xxlem4.3}
Let $A$ be a noetherian connected graded algebra, let
${\mathbb X}$ be $\Proj A$, and let ${\mathcal T}$ be 
$D^b(coh({\mathbb X}))$. 
\begin{enumerate}
\item[(1)]
\cite[Corollary 4.14]{Ye}
If $A$ is Artin-Schelter Gorenstein, then $A$ has a balanced 
dualizing complex.
\item[(2)]
\cite[Corollary 4.3]{YZ}
Suppose $A$ is Artin-Schelter Gorenstein such that ${\mathbb X}$
has finite homological dimension. Let $d={\rm{injdim}}\;
A$, $\ell$ be the AS index of $A$ and $\mu$
be the Nakayama automorphism of $A$. 
Then the Serre functor of ${\mathcal T}$ is $-\otimes_{\mathcal O}
\pi({^{\mu}A^{1}})(-\ell) [d-1]$.
\end{enumerate}
\end{lemma}

Let $M$ be a locally finite ${\mathbb Z}$-graded module or 
vector space. The Hilbert series of $M$ is defined to be
$$H_{M}(t):=\sum_{n\in {\mathbb Z}} \dim M_n t^n.$$
Let $A$ be a graded algebra and $s$ be a positive integer. The 
{\it $s$th Veronese subalgebra} of $A$ is defined to be
$$A^{(s)}:=\oplus_{n\in {\mathbb Z}} A_{sn}.$$
The following lemma is well-known.

\begin{lemma}
\label{xxlem4.4} 
Let $A$ be a noetherian connected graded algebra generated 
in degree 1. Let $s$ be a positive integer.
\begin{enumerate}
\item[(1)]
$A$ is a finitely generated module over $A^{(s)}$ on both 
sides.
\item[(2)]
$\GKdim A=\GKdim A^{(s)}$.
\item[(3)]
If the Hilbert series of $A$ is a rational function, then 
so is the Hilbert series of $A^{(s)}$. 
\item[(4)]
If the Hilbert series of $A$ is a rational function, 
then 
$$\limsup_{n\to \infty}\log_n (\dim A_n)=\GKdim A-1.$$
\item[(5)]
If $A$ is an Artin-Schelter regular algebra, then the 
Hilbert series of $A$ is a rational function.
\end{enumerate}
\end{lemma}

\begin{proof} (1) Connected graded noetherian algebras are 
finitely generated. So $A$ is generated by $\oplus_{i=0}^{s-1} A_i$
over $A^{(s)}$ on both sides.

(2) It follows from part (1) and \cite[Proposition 8.2.9(i)]{MR}.

(3) This follows from the fact that
$$H_A(t)=\frac{1}{s} \sum_{i=0}^{s-1} H_{A}(\xi^i t)$$
where $\xi$ is an $s$th primitive root of unity.

(4) When $H_A(t)$ is a rational function, $\dim A_n$ is a
multi-polynomial function of $n$ in the sense of 
\cite[p.399]{Zh}. Say $d$ is the degree of the multi-polynomial 
function $\dim A_n$ of $n$. Then $\GKdim A=d+1$ by using 
\cite[(E7)]{Zh}. This implies that 
$$\limsup_{n\to \infty}\log_n (\dim A_n)=d=\GKdim A-1.$$

(5) This is \cite[Proposition 3.1]{StZ}.
\end{proof}

\begin{theorem}
\label{xxthm4.5}
Let $A$ be a noetherian connected graded Artin-Schelter Gorenstein
algebra of injective dimension $d\geq 2$ that is generated in degree 1.
Suppose that ${\mathbb X}:=\Proj A$ has finite 
homological dimension. Let ${\mathcal T}$ be the bounded derived
category of $coh({\mathbb X})$. In parts (2,3,4,5,6), we further assume
that the Hilbert series of $A$ is rational. Let $\ell$ be the AS index 
of $A$.
\begin{enumerate}
\item[(1)]
$\fpcy({\mathcal T})=d-1$.
\item[(2)]
If $\ell>0$, then 
$\fpk({\mathcal T}_{\ast})=-\infty$ and 
$\fpk^{-1}({\mathcal T}_{\ast})=\GKdim A-1$.
\item[(3)]
If $\ell<0$, then 
$\fpk({\mathcal T}_{\ast})=\GKdim A-1$ and 
$\fpk^{-1}({\mathcal T}_{\ast})=-\infty$.
\item[(4)]
If $\ell=0$, then $\fpk({\mathcal T}_{\ast})=
\fpk^{-1}({\mathcal T}_{\ast})=0$.
\item[(5)]
For all objects ${\mathcal C}$ and ${\mathcal D}$
in ${\mathcal T}$,
$$\limsup_{n\to\infty}\;
\log_n (\dim \Hom_{\mathcal T}({\mathcal C}, 
(S\circ \Sigma^{-d})^{n}({\mathcal D})))\leq \GKdim A-1.
$$
\item[(6)]
For all objects ${\mathcal C}$ and ${\mathcal D}$
in ${\mathcal T}$,
$$\limsup_{n\to\infty}\;
\log_n (\dim \Hom_{\mathcal T}({\mathcal C}, 
(S\circ \Sigma^{-d})^{-n}({\mathcal D})))\leq \GKdim A-1.
$$
\end{enumerate}
\end{theorem}

\begin{proof} (1)
Let $\ell$ be the AS index of $A$.
There are two different cases: $\ell\leq 0$ and $\ell\geq 0$.
The proofs are similar, so we only prove the assertion for the 
first case.

First we claim that $Sp({\mathcal T})\subseteq (1, d-1){\mathbb Z}$.
Suppose $(m,n)\not\in (1, d-1){\mathbb Z}$. Let
$\sigma= S^m\circ \Sigma^{-n}=-\otimes {\mathcal M}^{\otimes m} [(d-1)m-n]$
where ${\mathcal M}=\pi({^{\mu}A^{1}})(-\ell)$ and where
$(d-1)m-n\neq 0$. Since
$\Proj A$ has finite global dimension, we have that, 
for all objects $A$ and $B$ in ${\mathcal T}$,
$$\Hom_{\mathcal T}(A, \sigma^t(B))=
\Hom_{\mathcal T}(A, (B\otimes {\mathcal M}^{\otimes tm}) [t((d-1)m-n)])
=0$$
for all $t\gg 0$. This implies that 
$\underline{\fpg}(\sigma)=-\infty$. Hence $(m,n)\not\in Sp({\mathcal T})$
and hence we have proven the claim.

Second we claim that $(1, d-1){\mathbb N}\subseteq Sp({\mathcal T})$.
Let ${\mathcal O}=\pi(A)$. It is a brick object as $\Hom_{\mathcal T}
({\mathcal O},{\mathcal O})=A_0=\Bbbk$. 
For every $(m,n)=(s, s(d-1)) \in (1, d-1){\mathbb N}$, let
$\sigma=S^m \circ \Sigma^{-n}=-\otimes {\mathcal M}^{\otimes s}$. Then 
$$\sigma({\mathcal O})\cong {\mathcal O}(-\ell s)$$
and 
$$\Hom_{\mathcal T}({\mathcal O}, \sigma({\mathcal O}))
\cong A_{-\ell s}\neq 0$$
when $\ell\leq 0$. 
This implies that $\rho(A(\phi,\sigma))\geq 1$ where 
$\phi=\{{\mathcal O}\}$. Similarly, $\rho(A(\phi,\sigma^n))\geq 1$ 
for all $n$. Consequently, $\underline{\fpg}(\sigma)\geq 0>-\infty$. 
Therefore $(m,n)\in Sp({\mathcal T})$ as desired. Now we have
$$(1, d-1){\mathbb N}\subseteq Sp({\mathcal T})\subseteq 
(1, d-1){\mathbb Z}$$
which implies that $\fpcy({\mathcal T})=d-1$.

(2) Let $w=\GKdim A$ which is at least 2. By Lemma 
\ref{xxlem4.4}(2,3,4), for every integer $s\geq 1$, 
\begin{equation}
\label{E4.5.1}\tag{E4.5.1}
\limsup_{n\to \infty}\log_n (\dim A_{sn})=\GKdim A-1=w-1.
\end{equation}

Now assume that $\ell$ is positive. By Lemma 
\ref{xxlem4.3}(2), $\sigma:=S \circ \Sigma^{-(d-1)}$ is equivalent to
$-\otimes {\mathcal O} (-\ell)$ when applied to ${\mathcal O}$.
Thus
\begin{equation}
\label{E4.5.2}\tag{E4.5.2}
\Hom_{\mathcal T}({\mathcal O}, \sigma^{n}({\mathcal O}))
=A_{-\ell n}=0
\end{equation} 
when $n\geq 0$ and 
\begin{equation}
\label{E4.5.3}\tag{E4.5.3}
\Hom_{\mathcal T}({\mathcal O}, \sigma^{-n}({\mathcal O}))
=A_{\ell n}
\end{equation} 
for $n\geq 0$. Now \eqref{E4.5.2} implies that
$\kappa({\mathcal T},{\mathcal O})=-\infty$, and 
\eqref{E4.5.3} together with \eqref{E4.5.1} implies
that $\kappa^{-1}({\mathcal T},{\mathcal O})=w-1=\GKdim A-1$.

(3) Similar to the proof of (2).

(4) Assume that $\ell=0$. Then 
$\sigma:=S \circ \Sigma^{-(d-1)}$ is equivalent to
$-\otimes {\mathcal O} (-\ell)$ when applied to ${\mathcal O}$.
Thus
\begin{equation}
\notag
\Hom_{\mathcal T}({\mathcal O}, \sigma^{n}({\mathcal O}))
=A_{-\ell n}=A_{0}=\Bbbk
\end{equation}
for all $n\in {\mathbb Z}$. This implies that 
$\fpk({\mathcal T}_{\ast})=
\fpk^{-1}({\mathcal T}_{\ast})=0$.

(5,6) The proofs are similar. We only consider (5).
Note that $S\circ \Sigma^{-(d-1)}=-\otimes \pi(^{\mu} A^1)(-\ell)$.
By \cite[Lemmas 4.2.3 and 4.3.2]{BV}, ${\mathcal T}$
is generated by $\{{\mathcal O}(n)\}_{n\in {\mathbb Z}}$.
Hence we can assume that ${\mathcal C}={\mathcal O}$ and 
${\mathcal D}={\mathcal O}(a)[b]$
for some $a$ and $b$. Note that \cite[Theorem 8.1(3)]{AZ}
holds for Artin-Schelter Gorenstein algebras. Then 
\cite[Theorem 8.1(3)]{AZ} implies that 
$$\dim \Ext^b_{\mathbb X} ({\mathcal O}, 
(S\circ \Sigma^{-(d-1)})^{\otimes n}({\mathcal O}(a)))\leq c n^{w-1}$$
for some constant $c$ only dependent on $a,b$. 
Therefore the assertion follows.
\end{proof}

The noncommutative projective scheme in the sense of \cite{AZ} 
can be defined for connected graded coherent algebras that are 
not necessarily noetherian. Here we consider a family of 
noncommutative projective schemes of non-noetherian 
Artin-Schelter regular algebras of global dimension two.

Let $W_n$ be the Artin-Schelter regular algebra 
$\Bbbk\langle x_1, \cdots, x_n\rangle/(\sum_{i=1}^n x_i^2)$ of 
global dimension 2. When $n\geq 3$, this algebra is 
non-noetherian \cite[Theorem 0.2(1)]{Zh}, but coherent 
\cite[Theorem 1.2]{Pi}. Let ${\mathbb P}^1_n$ denote the 
noncommutative projective scheme associated to $W_n$ defined 
in \cite{Pi}, which is also denoted by $\Proj W_n$. We call 
${\mathbb P}^1_n$ {\it a Piontkovski projective line} of 
rank $n$. See \cite{Pi, SSm} for basic properties of 
${\mathbb P}^1_n$. The main result concerning 
${\mathbb P}^1_n$ is the following. See \cite{CG} for the 
definition of $\fpc$ and $\fpv$. 

\begin{theorem}
\label{xxthm4.6} 
Let ${\mathbb P}^1_n$ be a Piontkovski projective line of 
rank $n\geq 2$. Let ${\mathcal T}_n$ be the derived category
$D^b(coh({\mathbb P}^1_{n}))$. Then
\begin{enumerate}
\item[(1)]
$\fpdim ({\mathcal T}_{n})=1$ for all $n\geq 2$.
\item[(2)]
$\fpgldim ({\mathcal T}_{n})=1$ for all $n\geq 2$.
\item[(3)]
$\fpc ({\mathcal T}_{n})=0$ for all $n\geq 2$.
\item[(4)]
$\fpv ({\mathcal T}_{n})=0$ for all $n\geq 2$.
\item[(5)]
$\fpcy ({\mathcal T}_{n})=1$ for all $n\geq 2$.
\item[(6)]
$\fpk^{-1} (({\mathcal T}_{n})_{\ast})=\begin{cases} 1 & n=2,\\
\infty & n\geq 3.\end{cases}$
\item[(7)]
$\fpk(({\mathcal T}_{n})_{\ast})=-\infty$ for all $n\geq 2$.
\end{enumerate}
\end{theorem}

\begin{proof} Let ${\mathcal O}$ denote the object $\pi(W_n)$ in 
${\mathbb X}:=\Proj W_n$. Then $coh({\mathbb X})$ is hereditary, 
and ${\mathcal T}_n$ is Ext-finite with Serre functor $S:=-\otimes 
{\mathcal O}(-2)[1]$ \cite[Proposition 1.5]{Pi}.

(1) By \cite[Theorem 3.5(4)]{CG}, $\fpdim ({\mathcal T}_{n})=
\fpdim (coh({\mathbb X}))$. Since $coh({\mathbb X})$ is hereditary,
$\fpdim (coh({\mathbb X}))\leq 1$. It remains to show that
$\fpdim (coh({\mathbb X}))\geq 1$. Let $A$ be the simple object in
$coh({\mathbb X})$ of the form $\pi(W/I))$ where
$I$ is the right ideal of $W$ generated by $W(x_3,\cdots,x_n)+x_2$.
Then it is routine to verify that $\Ext^1_{coh({\mathbb X})}
(A,A)=\Bbbk$. This implies that $\fpdim (coh({\mathbb X}))\geq 1$ 
as required.

(2) By \cite[Theorem 3.5(1)]{CG}, $\fpgldim ({\mathcal T}_n)
\leq \gldim coh({\mathbb X})=1$. By part (1), $\fpgldim ({\mathcal T}_n)
\geq 1$. The assertion follows.

(3,4) These follow from the fact that $coh({\mathbb X})$ 
is hereditary.

(5) Since the Serre functor $S$ is of the form $-\otimes 
{\mathcal O}(-2)[1]$, we can almost copy the proof of 
Proposition \ref{xxpro3.3}.

(6,7) Using the special form of the Serre functor $S$,
we can adapt the proofs of Proposition \ref{xxpro3.6} and 
Theorem \ref{xxthm4.5}.
%
\end{proof}

\section{Comments on fp-invariants of finite dimensional algebras}
\label{xxsec5}

In this section we give some remarks, comments and examples
concerning finite dimensional algebras. Let ${\mathcal T}(A)$ 
be the derived category $D^b(\Mod_{f.d}-A)$. The fp
Calabi-Yau dimension $\fpcy({\mathcal T})$ is defined as in
Definition \ref{xxdef3.1}. 

\begin{example} 
\label{xxex5.1}
Let $Q$ be a finite acyclic quiver and 
$A$ be the path algebra $\Bbbk Q$. Let ${\mathcal T}$ be the 
bounded derived category $D^b(\Mod_{f.d}-A)$. 
\begin{enumerate}
\item[(1)]
If $Q$ is of ADE type, then ${\mathcal T}$ is fractional
Calabi-Yau and by Lemma \ref{xxlem3.2}(3) and 
\cite[Example 8.3(2)]{Ke1},
$$\fpcy({\mathcal T})=\frac{h-2}{h}$$
where $h$ is the Coxeter number of $Q$. In this case, 
$\fpcy({\mathcal T})$ is strictly between 0 and $1$.
\item[(2)]
If $Q$ is of $\widetilde{A}\widetilde{D}\widetilde{E}$ type,
then, by using Lemma \ref{xxlem2.1}(2), ${\mathcal T}$
is equivalent to $D^b(coh({\mathbb X}))$ for a 
weighted projective line ${\mathbb X}$. Similar to the 
proof of Theorem \ref{xxthm4.5}(1), one can show that 
$$\fpcy(D^b(coh({\mathbb X})))=1$$
for every weighted projective line ${\mathbb X}$. 
(The proof is slightly more complicated and details 
are omitted). Thus we obtain that 
$$\fpcy({\mathcal T})=1.$$
\end{enumerate}
\end{example}

If $Q$ is of wild representation type, in some cases,
one can shows that
\begin{equation}
\label{E5.1.1}\tag{E5.1.1}
\fpcy({\mathcal T})=1.
\end{equation}
Such an example is the Kronecker quiver with two vertices 
and $n$ arrows from the first vertex to the second, see
Example \ref{xxex5.7}. It is not clear to us
if \eqref{E5.1.1} holds for all acyclic quivers of wild 
representation type.

\begin{lemma}
\label{xxlem5.2} 
Let $A$ and $B$ be finite dimensional algebras of 
finite global dimension. Suppose that both 
${\mathcal T}(A)$ and ${\mathcal T}(B)$ are fractional
Calabi-Yau of dimension $d_1$ and $d_2$ respectively. 
Then ${\mathcal T}(A\otimes B)$ is fractional
Calabi-Yau of dimension $d_1+d_2$.
\end{lemma}

The proof of Lemma \ref{xxlem5.2} is straightforward and
hence omitted. One immediate question is

\begin{question}
\label{xxque5.3} 
Let $A$ and $B$ be finite dimensional algebras of 
finite global dimension. Is then
$$\fpcy({\mathcal T}(A\otimes B)) =\fpcy ({\mathcal T}(A))
+\fpcy({\mathcal T}(B))?$$
\end{question}

To define fp (anti-)Kodaira 
dimension of ${\mathcal T}$ as in Definition \ref{xxdef3.5}, 
we need to specify an object ${\mathcal O}$ which plays the 
role of the structure sheaf in algebraic geometry. One 
choice for ${\mathcal O}$ is the left $A$-module $A$. 
So we let ${\mathcal T}(A)_{\ast}$ be $({\mathcal T}(A), A)$.

\begin{definition}
\label{xxdef5.4}
Let $A$ be a finite dimensional algebra of finite global 
dimension. Suppose 
that ${\mathcal T}(A)$ has a fractional fp Calabi-Yau 
dimension $\frac{a}{b}\in {\mathbb Q}$. 
\begin{enumerate}
\item[(1)]
The {\it fp Kodaira  dimension} of $A$ 
is defined to be
$$\fpk(A)
:=\fpk({\mathcal T}(A)_{\ast})
=\lim_{t} \left\{ \limsup_{n\to \infty} 
\dim \Hom_{\mathcal T}(A, (S^{bnt}\circ \Sigma^{-ant})
(A))\right\}$$
where the first limit ranges over all positive integers
$t$ with order $<$ as defined before Definition \ref{xxdef3.5}.
\item[(2)]
The {\it fp anti-Kodaira  dimension} of 
$A$ is defined to be
$$\fpk^{-1}(A)
:=\fpk^{-1}({\mathcal T}(A)_{\ast})
=\lim_{t} \left\{ \limsup_{n\to \infty} 
\dim \Hom_{\mathcal T}(A, (S^{-bnt}\circ \Sigma^{ant})
(A))\right\}$$
where the first limit ranges over all positive integers
$t$ with order $<$ as defined before Definition \ref{xxdef3.5}.
\end{enumerate}
\end{definition}

\begin{remark}
\label{xxrem5.5} 
Suppose ${\mathcal T}(A)$ is fractional Calabi-Yau.
Then one can check easily that 
\begin{equation}
\label{E5.5.1}\tag{E5.5.1}
\fpk(A)=\fpk^{-1}(A)=0.
\end{equation}
So if there is a notion of a fractional Calabi-Yau algebra,
\eqref{E5.5.1} should be a part of the definition.
\end{remark}

Let $A$ be a finite dimensional algebra of finite global
dimension. Then the Serre functor is given by $-\otimes_A^L A^{\ast}$. 
It is unknown if $\fpcy({\mathcal T}(A))$ always exists. If 
$\fpcy({\mathcal T}(A))$ exists and is a rational number, 
then we can define and calculate fp Kodaira dimension (respectively, 
anti-Kodaira dimension) of $A$. Here is a list of questions 
that are related to the fp Kodaira dimension.

\begin{question}
\label{xxque5.6}
Let $A$ and $B$ be two finite dimensional algebra of finite 
global dimension. Suppose that $\fpcy({\mathcal T}(A))$ and 
$\fpcy({\mathcal T}(B))$ are rational numbers.
\begin{enumerate}
\item[(1)]
Are $\fpk(A)$ and $\fpk^{-1}(A)$ less than $\infty$?
\item[(2)]
If ${\mathcal T}(A)$ is triangulated equivalent to 
${\mathcal T}(B)$, is then $\fpk(A)=\fpk(B)$ 
and $\fpk^{-1}(A)=\fpk^{-1}(B)$?
\item[(3)]
Suppose both $\fpk(A)$ and $\fpk(B)$ are finite, is then
$$\fpk(A\otimes B)=\fpk(A)+\fpk(B)?$$
\end{enumerate}
\end{question}

The first question has a negative answer.

\begin{example}
\label{xxex5.7} 
Let $Q_{n}$ be the Kronecker quiver with two vertices and $n$ 
arrows from the first vertex to the second. Let $U_{n}$ be 
the path algebra of $Q_{n}$. By \cite[Theorem 0.1]{Min}, if 
$n\geq 2$, 
\begin{equation}
\label{E5.7.1}\tag{E5.7.1}
D^b(Mod_{f.d.}-U_{n})\cong D^{b}(coh({\mathbb P}^1_{n}))=:
{\mathcal T}_n
\end{equation}
where ${\mathbb P}^1_n$ is given in Theorem \ref{xxthm4.6}.
By Theorem \ref{xxthm4.6}(3), $\fpcy(D^b(Mod_{f.d.}-U_{n}))
=1$ for all $n\geq 2$. On the other hand, 
$\fpcy(D^b(Mod_{f.d.}-U_{1}))= \frac{h-2}{h}=\frac{1}{3}$ where 
$h=3$ is the Coxeter number of the quiver $A_2$ (which is $Q_1$), 
see \cite[Example 8.3(2)]{Ke1}. 

We claim that $\fpk(U_n)=-\infty$ and that 
$\fpk^{-1}(U_n)=\infty$ when $n\geq 3$. We only prove the 
second assertion. By the noncommutative Beilinson's theorem
given in \cite[Theorem 0.1]{Min} (also see \eqref{E5.9.1}), 
the equivalent \eqref{E5.7.1} 
sends $U_n$ to ${\mathcal O}\oplus {\mathcal O}(1)$ where 
${\mathcal O}$ is the structure sheaf of ${\mathbb P}^1_{n}$.

Note that the Serre functor is $S:=-\otimes 
{\mathcal O}(-2)[1]$. Let $A=U_n$. Then 
$$\begin{aligned}
\fpk^{-1}(A)&=\lim_{t} \left\{ \limsup_{m\to \infty} 
\dim \Hom_{{\mathcal T}(A)}(A, (S^{-tm}\circ \Sigma^{tm})
(A))\right\}\\
&=\lim_{t} \left\{ \limsup_{m\to \infty} 
\dim \Hom_{{\mathcal T}_n}({\mathcal O}\oplus {\mathcal O}(1), 
(S^{-tm}\circ \Sigma^{tm}) ({\mathcal O}\oplus {\mathcal O}(1)))\right\}\\
&\geq \lim_{t} \left\{ \limsup_{m\to \infty} 
\dim \Hom_{{\mathcal T}_n}({\mathcal O}, (S^{-tm}\circ \Sigma^{tm})
({\mathcal O}))\right\}\\
&=\fpk^{-1}(({\mathcal T}_{n})_{\ast})\\
&=\infty
\end{aligned}
$$
where the last equation is Theorem \ref{xxthm4.6}(6).
%
%
\end{example}

We add a few more questions to Question \ref{xxque0.4}.

\begin{question}
\label{xxque5.8} Let $A$ be a finite dimensional algebra of finite
global dimension and let ${\mathcal T}$ be the derived category
$D^b(\Mod_{f.d}-A)$.
\begin{enumerate}
\item[(1)]
By Example \ref{xxex5.7}, if $A$ is the path algebra of the 
quiver $Q_1$, then $\fpcy({\mathcal T})=\frac{1}{3}$. Is the 
minimum value of $\fpcy({\mathcal T})$ equal to $\frac{1}{3}$
for an arbitrary $A$?
\item[(2)]
Is there a value of $\fpcy({\mathcal T})$ outside of the set
$$R:=\left(\sum_{h\geq 3} \frac{h-2}{h} {\mathbb N}\right) \cap {\mathbb Q}_{>0}?$$
By Lemma \ref{xxlem5.2} and \cite[Example 8.3(2)]{Ke1}, every 
number in $R$ can be realized as $\fpcy({\mathcal T})$ 
for some finite dimensional algebra. But we don't have examples of
$\fpcy$ that are outside this range.
\end{enumerate}
\end{question}

Recall that 

\begin{definition} \cite[p. 1230]{HP}
\label{xxdef5.9}
Let ${\mathbb X}$ be a smooth projective scheme.
\begin{enumerate}
\item[(1)]
A coherent sheaf ${\mathcal E}$ on ${\mathbb X}$ is called 
{\it exceptional} if $\Hom_{\mathbb X}({\mathcal E},{\mathcal E}) 
\cong \Bbbk$ and $\Ext^i_{\mathbb X}({\mathcal E},{\mathcal E}) =0$
for every $i \geq  0$. 
\item[(2)]
A sequence ${\mathcal E}_1, \cdots, {\mathcal E}_n$ of exceptional 
sheaves is called an {\it exceptional sequence} if $\Ext^k_{\mathbb X}
({\mathcal E}_i,{\mathcal E}_j) = 0$ for all $k$ and for all $i > j$. 
\item[(3)]
If an exceptional sequence generates $D^b(coh({\mathbb X}))$, then 
it is called {\it full}. 
\item[(4)]
If an exceptional sequence satisfies
$$\Ext^k_{\mathbb X}({\mathcal E}_i,{\mathcal E}_j) = 0$$ 
for all $k > 0$ and all $i, j$, then it is called a {\it strongly 
exceptional sequence}.
\end{enumerate}
\end{definition}

The existence of a full exceptional sequence has been proved for many 
projective schemes. However, on Calabi-Yau varieties there are no 
exceptional sheaves. When ${\mathbb X}$ has a full exceptional sequence
${\mathcal E}_1, \cdots, {\mathcal E}_n$, then there is a triangulated 
equivalence
\begin{equation}
\label{E5.9.1}\tag{E5.9.1}
D^b(coh({\mathbb X}))\cong D^b(\Mod_{f.d}-A)
\end{equation} 
where $A$ is the finite dimensional algebra
$\End_{\mathbb X}(\oplus_{i=1}^n {\mathcal E}_i)$. In this 
setting, the fp Calabi-Yau dimension of $D^b(\Mod_{f.d}-A)$ is
equal to $\dim {\mathbb X}$, which exists and is finite. 
In many examples in algebraic geometry, a full 
exceptional sequence consists of line bundles. 
Assume this is true. Via \eqref{E5.9.1}, one sees easily that 
$\fpk^{\pm 1}(A)\geq \fpk^{\pm 1}({\mathbb X})$. In fact, in many examples,
we have $\fpk^{\pm 1}(A)= \fpk^{\pm 1}({\mathbb X})$.

\section{Appendix A: Proof of Lemma \ref{xxlem2.10}}
\label{xxsec6}

As in Section \ref{xxsec2}, let $r$ be a positive 
integer. Suppose that a hereditary abelian category 
${\mathfrak T}$ has $r^2$ brick objects as given in 
Corollary \ref{xxcor2.8}, now labeled as 
$$1,2, 3, \ldots, r^2-1, r^2,$$ 
where the matrices 
$$H=\left(H_{ij}\right)_{r^2\times r^2}=
\left(\dim \Hom (i,j)\right)_{r^2\times r^2},$$
see \eqref{E2.8.1}, and
$$E=\left(E_{ij}\right)_{r^2\times r^2}=
\left(\dim \Hom(i,\Sigma j)\right)_{r^2\times r^2},$$ 
see \eqref{E2.8.2}, are given by the following block 
matrices:

\begin{equation} \label{E6.0.1}\tag{E6.0.1}
H=
\begin{blockarray}{ccccccc}
{} & (1) & (2) & (3) & (4) & \cdots & (r) \\
\begin{block}{c(cccccc@{\hspace*{5pt}})}
(1) & P^0 & P^1 & P^2 & P^3 & \cdots & P^{r-1} \\
(2) & P^0 & \sum\limits_{i=0}^1 P^i & \sum\limits_{i=1}^2 P^i 
    & \sum\limits_{i=2}^3 P^i & \cdots 
		& \sum\limits_{i=r-2}^{r-1} P^i \\
(3) & P^0 & \sum\limits_{i=0}^1 P^i &  \sum\limits_{i=0}^2 P^i 
    & \sum\limits_{i=1}^3 P^i & \cdots 
		& \sum\limits_{i=r-3}^{r-1} P^i	 \\
(4) & P^0 & \sum\limits_{i=0}^1 P^i  & \sum\limits_{i=0}^2 P^i 
    & \sum\limits_{i=0}^3 P^i & \cdots 
		& \sum\limits_{i=r-4}^{r-1} P^i  \\
	\vdots&\vdots&\cdots&\cdots&\cdots&\cdots&\vdots \\
(r) & P^0 & \sum\limits_{i=0}^1 P^i & \sum\limits_{i=0}^2 P^i 
& \sum\limits_{i=0}^3 P^i & \cdots & \sum\limits_{i=0}^{r-1} P^i \\
\end{block}
\end{blockarray},
\end{equation}
	
\begin{align}
E&=
\begin{blockarray}{ccccccc}\notag
{} & (1) & (2) & (3) & (4) & \cdots & (r) \\
\begin{block}{c(cccccc@{\hspace*{5pt}})}
(1) & P^{r-1} & P^{r-1} & P^{r-1} & P^{r-1} 
    & \cdots & P^{r-1} \\
(2) & P^{r-2} & \sum\limits_{i=r-2}^{r-1} P^i 
    & \sum\limits_{i=r-2}^{r-1} P^i 
		& \sum\limits_{i=r-2}^{r-1} P^i 
		& \cdots & \sum\limits_{i=r-2}^{r-1} P^i \\
(3) & P^{r-3} & \sum\limits_{i=r-3}^{r-2} P^i 
    & \sum\limits_{i=r-3}^{r-1} P^i 
		& \sum\limits_{i=r-3}^{r-1} P^i 
		& \cdots & \sum\limits_{i=r-3}^{r-1} P^i \\
(4) & P^{r-4} & \sum\limits_{i=r-4}^{r-3} P^i 
    & \sum\limits_{i=r-4}^{r-2} P^i 
		& \sum\limits_{i=r-4}^{r-1} P^i 
		& \cdots & \sum\limits_{i=r-4}^{r-1} P^i \\
\vdots&\vdots&\cdots&\cdots&\cdots&\cdots&\vdots \\
(r) & P^{0} & \sum\limits_{i=0}^1 P^i 
    & \sum\limits_{i=0}^2 P^i 
		& \sum\limits_{i=0}^3 P^i 
		& \cdots & \sum\limits_{i=0}^{r-1} P^i \\
\end{block}
\end{blockarray} \\
&=\begin{blockarray}{ccccccc} \notag
{} & (1) & (2) & (3) & (4) & \cdots & (r) \\
\begin{block}{c(cccccc@{\hspace*{5pt}})}
(1) & P^{r-1} & P^{r-1} & P^{r-1} & P^{r-1} 
    & \cdots & P^{r-1} \\
(2) & P^{r-2} & \sum\limits_{i=1}^{2} P^{r-i} 
    & \sum\limits_{i=1}^{2} P^{r-i} 
		& \sum\limits_{i=1}^{2} P^{r-i} 
		& \cdots & \sum\limits_{i=1}^{2} P^{r-i} \\
(3) & P^{r-3} & \sum\limits_{i=2}^{3} P^{r-i} 
    & \sum\limits_{i=1}^{3} P^{r-i} 
		& \sum\limits_{i=1}^{3} P^{r-i} 
		& \cdots & \sum\limits_{i=1}^{3} P^{r-i} \\
(4) & P^{r-4} & \sum\limits_{i=3}^{4} P^{r-i} 
    & \sum\limits_{i=2}^{4} P^{r-i} 
		& \sum\limits_{i=1}^{4} P^{r-i} 
		& \cdots & \sum\limits_{i=1}^{4} P^{r-i} \\
\vdots&\vdots&\cdots&\cdots&\cdots&\cdots&\vdots \\
(r) & P^{0} & \sum\limits_{i=r-1}^r P^{r-i} 
    & \sum\limits_{i=r-2}^r P^{r-i} 
		& \sum\limits_{i=r-3}^r P^{r-i} 
		& \cdots & \sum\limits_{i=1}^{r} P^{r-i} \\
\end{block}
\end{blockarray},
\end{align}
where $P$ is the $r \times r$ permutation matrix
$$
\begin{pmatrix}
	0 & 0 & 0 & 0 & 0 & 1 \\
	1 & 0 & 0 & 0 & 0 & 0 \\
	0 & 1 & 0 & 0 & 0 & 0 \\
	0 & 0 & 1 & 0 & 0 & 0 \\
	\vdots & \ddots & \ddots & \ddots & \ddots & \vdots \\
	0 & 0 & 0 & 0 & 1 & 0 \\
\end{pmatrix},
$$
and $P^0$ denotes the $r \times r$ identity matrix. 
Note that $P^r=P^0$. We will show that 
$\rho(A(\phi)) \leq 1$ for all brick sets $\phi$. 

The Hom and $\Ext^1$ matrices given in \eqref{E2.8.1}-\eqref{E2.8.2}
are actually the transpose of the usual Hom and $\Ext^1$ matrices 
since they follow the convention of Corollary \ref{xxcor2.8}.
See the remark before Theorem \ref{xxthm2.11}.

First, notice that 
\begin{equation} 
\notag
H^T=\begin{blockarray}{ccccccc}
{} & (1) & (2) & (3) & (4) & \cdots & (r) \\
\begin{block}{c(cccccc@{\hspace*{5pt}})}
(1) & P^{r} & P^{r} & P^{r} & P^{r} & \cdots & P^{r} \\
(2) & P^{r-1} & \sum\limits_{i=0}^{1} P^{r-i} 
    & \sum\limits_{i=0}^{1} P^{r-i} 
		& \sum\limits_{i=0}^{1} P^{r-i} 
		& \cdots & \sum\limits_{i=0}^{1} P^{r-i} \\
(3) & P^{r-2} & \sum\limits_{i=1}^{2} P^{r-i} 
    & \sum\limits_{i=0}^{2} P^{r-i} 
		& \sum\limits_{i=0}^{2} P^{r-i} 
		& \cdots & \sum\limits_{i=0}^{2} P^{r-i} \\
(4) & P^{r-3} & \sum\limits_{i=2}^{3} P^{r-i} 
    & \sum\limits_{i=1}^{3} P^{r-i} 
		& \sum\limits_{i=0}^{3} P^{r-i} 
		& \cdots & \sum\limits_{i=0}^{3} P^{r-i} \\
\vdots&\vdots&\cdots&\cdots&\cdots&\cdots&\vdots \\
(r) & P^{1} & \sum\limits_{i=r-2}^{r-1} P^{r-i} 
    & \sum\limits_{i=r-3}^{r-1} P^{r-i} 
		& \sum\limits_{i=r-4}^{r-1} P^{r-i} 
		& \cdots & \sum\limits_{i=0}^{r-1} P^{r-i} \\
\end{block}
\end{blockarray}.
\end{equation}

Define the following non-negative $r^2 \times r^2$ matrices:
\[
F:=\begin{blockarray}{ccccccc}
    	{} & (1) & (2) & (3) & (4) & \cdots & (r) \\
    	\begin{block}{c(cccccc@{\hspace*{5pt}})}
    	(1) & P^{r-1} & P^{r-1} & P^{r-1} & P^{r-1} & \cdots & P^{r-1} \\
    	(2) & P^{r-2} & P^{r-2} & P^{r-2} & P^{r-2} & \cdots & P^{r-2} \\
	(3) & P^{r-3} & P^{r-3} & P^{r-3} & P^{r-3} & \cdots & P^{r-3} \\
	(4) & P^{r-4} & P^{r-4} & P^{r-4} & P^{r-4} & \cdots & P^{r-4} \\
	\vdots&\vdots&\cdots&\cdots&\cdots&\cdots&\vdots \\
    	(r) & P^{0} & P^{0} & P^{0} & P^{0} & \cdots & P^{0} \\
    	\end{block}
  	\end{blockarray},
\]
\[
G:=\begin{blockarray}{ccccccc}
    	{} & (1) & (2) & (3) & (4) & \cdots & (r) \\
    	\begin{block}{c(cccccc@{\hspace*{5pt}})}
    	(1) & P^{r} & P^{r} & P^{r} & P^{r} & \cdots & P^{r} \\
    	(2) & P^{r-1} & P^{r} & P^{r} & P^{r} & \cdots & P^{r} \\
	(3) & P^{r-2} & P^{r-1} & P^{r} & P^{r} & \cdots & P^{r} \\
	(4) & P^{r-3} & P^{r-2} & P^{r-1} & P^{r} & \cdots & P^{r} \\
	\vdots&\vdots&\cdots&\cdots&\cdots&\cdots&\vdots \\
    	(r) & P^{1} & P^{2} & P^{3} & P^{4} & \cdots & P^{r} \\
    	\end{block}
  	\end{blockarray} .
\]

We can see that 

\begin{align} \label{E6.0.2}\tag{E6.0.2}
&H^T+F= \\
\notag
&\begin{blockarray}{ccccccc}
{} & (1) & (2) & (3) & (4) & \cdots & (r) \\
\begin{block}{c(cccccc@{\hspace*{5pt}})}
(1) & \sum\limits_{i=0}^{1} P^{r-i}  
    &  \sum\limits_{i=0}^{1} P^{r-i} 
		&  \sum\limits_{i=0}^{1} P^{r-i} 
		&  \sum\limits_{i=0}^{1} P^{r-i} 
		& \cdots &  \sum\limits_{i=0}^{1} P^{r-i}\\
(2) &  \sum\limits_{i=1}^{2} P^{r-i} 
    & \sum\limits_{i=0}^{2} P^{r-i} 
		&  \sum\limits_{i=0}^{2} P^{r-i} 
		&  \sum\limits_{i=0}^{2} P^{r-i} 
		& \cdots &  \sum\limits_{i=0}^{2} P^{r-i} \\
(3) &  \sum\limits_{i=2}^{3} P^{r-i} 
    & \sum\limits_{i=1}^{3} P^{r-i} 
		& \sum\limits_{i=0}^{3} P^{r-i} 
		& \sum\limits_{i=0}^{3} P^{r-i} 
		& \cdots & \sum\limits_{i=0}^{3} P^{r-i} \\
(4) &  \sum\limits_{i=3}^{4} P^{r-i} 
    & \sum\limits_{i=2}^{4} P^{r-i} 
		& \sum\limits_{i=1}^{4} P^{r-i} 
		& \sum\limits_{i=0}^{4} P^{r-i} 
		& \cdots & \sum\limits_{i=0}^{4} P^{r-i} \\
\vdots&\vdots&\cdots&\cdots&\cdots&\cdots&\vdots \\
(r) &  \sum\limits_{i=r-1}^{r} P^{r-i} 
    & \sum\limits_{i=r-2}^{r} P^{r-i} 
		& \sum\limits_{i=r-3}^{r} P^{r-i} 
		& \sum\limits_{i=r-4}^{r} P^{r-i} 
		& \cdots & \sum\limits_{i=0}^{r} P^{r-i} \\
\end{block}
\end{blockarray}\\
\notag
&=E+G.
\end{align}

To finish the proof we need a few lemmas.

\begin{lemma} 
\label{xxlem6.1}
Let $\phi$ be a brick set which is a subset of 
$\{1,2,\cdots, r^2\}$. Suppose that $I,J \in 
\phi$ and that the $(I,J)$-entry $E_{IJ}\neq 0$. 
Then $F_{IJ}=E_{IJ}=1$.
\end{lemma}

\begin{proof}
Since $E$ is a matrix 
whose entries consist of zeros and ones, $E_{IJ}=1$. 
We consider two different cases.

\noindent
{\bf Case 1: $I=J$.} $\quad$ By \eqref{E6.0.2},
\begin{align*}
1+F_{IJ}	&=(H^T)_{IJ}+F_{IJ}=E_{IJ}+G_{IJ} =E_{IJ}+1
\end{align*}
which implies that $F_{IJ}=E_{IJ}=1$.

\noindent
{\bf Case 2: $I\not=J$.} $\quad$ Since $I,J$ are 
distinct elements in the same brick set, we have 
$(H^T)_{IJ}=0$. By \eqref{E6.0.2},
\begin{align*}
F_{IJ}	&=(H^T)_{IJ}+F_{IJ}=E_{IJ}+G_{IJ} \geq 1.
\end{align*}
Since $F$ is a matrix whose entries are contained 
in $\{0,1\}$, $F_{IJ}=1.$
\end{proof}
	


\begin{lemma} 
\label{xxlem6.2}
If $\phi$ is a brick set, for each row $I$ 
there exists at most one column $J$ such 
that $A(\phi)_{IJ}\not =0$. If there exists 
a row $J$ such that $A(\phi)_{IJ}\not=0$, 
then $A(\phi)_{IJ}=1$.
\end{lemma}

\begin{proof} Assume for contradiction that 
we have a brick set $\phi$ containing elements 
$I, J, J'$ such that $J \not=J'$ and 
$E_{IJ}\not=0 \not=E_{IJ'}.$ 
Notice that we have the following general 
formula for $n,m \in \{0, \ldots, r-1 \}, 
i,j \in \{1, \ldots, r\}$:
\begin{equation} 
\notag
F_{nr+i, mr+j}=(P^{r-1-n})_{ij}.
\end{equation}
We can always write 
\begin{equation}\notag
I=nr+i, \quad J=mr+j, \quad J'=m'r+j', 
\end{equation}
for some $i,j,j' \in \{ 1, \ldots, r \}, n,m,m' 
\in \{0, \ldots, r-1\}$.

Assume without loss of generality that $m \geq m'$. 
We will show that
\begin{equation}
\notag
H_{JJ'}=H_{mr+j,m'r+j'} \not =0.
\end{equation}
Therefore, $J,J'$ cannot be in the same brick set, 
a contradiction.

By Lemma \ref{xxlem6.1},
\begin{align*}
&1= F_{IJ} = F_{nr+i,mr+j}=(P^{r-1-n})_{ij}=(P^{n+1-r})_{ji}, \\
&1 = F_{IJ'}=F_{nr+i,m'r+j'}=(P^{r-1-n})_{ij'}.
\end{align*}
Then
\begin{align*}
1 &= (P^{n+1-r})_{ji}(P^{r-1-n})_{ij'} 
= \sum_{k=1}^r (P^{n+1-r})_{jk}(P^{r-1-n})_{kj'} \\
&=(P^{n+1-r}P^{r-1-n})_{jj'} 
=\delta_{jj'} 
\end{align*}
which implies that $j=j'$.

By examination of the $H$ matrix \eqref{E6.0.1}, we can 
see that for $m \geq m'$,
\[
H_{JJ'}=H_{mr+j,m'r+j'}=(P^0)_{jj'}+A_{jj'}
\]
for some non-negative $r \times r$ matrix $A$. Therefore,
\[
H_{JJ'}=H_{mr+j,m'r+j'}=H_{mr+j,m'r+j}
=(P^0)_{jj}+A_{jj}=\delta_{jj}+A_{jj} \geq 1
\]
as required.
\end{proof}


\begin{lemma}\label{xxlem6.3}
If $\phi$ is a brick set, for each column $J$ there 
exists at most one row $I$ such that 
$A(\phi)_{IJ}\not =0$. If there exists a row $I$ 
such that $A(\phi)_{IJ}\not=0$, then $A(\phi)_{IJ}=1$.
\end{lemma}

\begin{proof} The proof is similar to the proof of Lemma 
\ref{xxlem6.2} and omitted.
%
%
%
%
%
\end{proof}


\begin{lemma}
\label{xxlem6.4}
If $\phi$ is a brick set, $\rho(A(\phi)) \leq 1$.
\end{lemma}

\begin{proof}
By Lemmas \ref{xxlem6.2} and \ref{xxlem6.3}, we have 
shown that if $\phi$ is a brick set, $A(\phi)$ is a 
matrix with at most one non-zero entry in each row and 
at most one non-zero entry in each column, such that 
if any entry is non-zero, it is $1$. Then 
$A(\phi)$ is almost a permutation matrix. In this
case the quiver corresponding to $A(\phi)$ has 
cycle number at most 1. By \cite[Theorem 1.8]{CG},
$\rho(A(\phi)) \leq 1$.
\end{proof}

\section{Appendix B: Some variants}
\label{xxsec7}

Recall that the set of subsets of $n$ nonzero objects 
in ${\mathcal C}$ is denoted by $\Phi_n$ for each $n\geq 1$.
And let $\Phi=\bigcup_{n\geq 1} \Phi_n$. In this paper, we
use either $\Phi_b$ or $\Phi_a$ as testing objects
in the definition of fp-invariants [Definition \ref{xxdef1.4}].
Depending on the situation, we might want to choose a testing set different
from $\Phi_b$ or $\Phi_a$. Here is a list of possible alternative 
testing sets.

\begin{example}
\label{xxex7.1}
\begin{enumerate}
\item[(1)]
$\Phi=\bigcup_{n\geq 1} \Phi_n$.
\item[(2)]
If the category ${\mathcal C}$ is abelian, we can consider 
``simple sets'' as follows. Let $\Phi_{n,s}$ be the set of 
$n$-object subsets of ${\mathcal C}$, say 
$\phi:=\{X_1, X_2, \cdots,X_n\}$, where the $X_i$ are non-isomorphic
simple objects in ${\mathcal C}$. 
Let $\Phi_{s}=\bigcup_{n\geq 1} \Phi_{n,s}$. 
\item[(3)]
A subset $\phi=\{X_1, X_2, \cdots,X_n\}$ is called 
{\it a triangular brick set} if each $X_i$ is a brick
object and, up to a permutation, $\Hom_{\mathcal C}(X_i,X_j)=0$
for all $i<j$. Let $\Phi_{n,tb}$ be the set of all triangular 
brick $n$-sets, and let $\Phi_{tb}=\bigcup_{n\geq 1} \Phi_{n,tb}$.
\item[(4)]
Now assume that ${\mathcal C}$ is a triangulated category with 
suspension functor $\Sigma$.
A subset $\phi=\{X_1, X_2, \cdots,X_n\}$ is called 
{\it a triangular atomic set} if each $X_i$ is an atomic
object and, up to a permutation, $\Hom_{\mathcal C}(X_i,X_j)=0$
for all $i<j$. Let $\Phi_{n,ta}$ be the set of all triangular 
atomic $n$-sets, and let $\Phi_{ta}=\bigcup_{n\geq 1} \Phi_{n,ta}$.
\end{enumerate}
\end{example}

Basically, for any property $P$, we can define $\Phi_{n,Pb}$
(respectively, $\Phi_{n,Pa}$) and let $\Phi_{Pb}$ 
(respectively, $\Phi_{Pa}$) be $\bigcup_{n\geq 1} 
\Phi_{n,Pb}$ (respectively, $\bigcup_{n\geq 1} \Phi_{n,Pa}$).
All of the definitions in this paper and in \cite{CG} can be modified
after we redefine $\rho(A(\phi,\sigma))$ as follows. 

\begin{definition}
\label{xxdef7.2} Let ${\mathcal C}$ be a $\Bbbk$-linear category 
and let $\phi$ be a set of $n$ nonzero objects, say
$\{X_1,\cdots,X_n\}$, in ${\mathcal C}$. Let $\sigma$ be a 
$\Bbbk$-linear endofunctor of ${\mathcal C}$. We define
$$\rho(A(\phi,\sigma)):=\frac
{\rho\left( (\dim \Hom_{\mathcal C}(X_i,\sigma(X_j)))_{n\times n}
\right)}
{\rho\left( (\dim \Hom_{\mathcal C}(X_i,X_j))_{n\times n}\right)}.
$$
\end{definition}

Note that $\rho(A(\phi,\sigma))$ agrees with the original definition
when $\phi$ is a brick set. One reason to introduce $P$-versions of 
fp-invariants is to extend these invariants even if the category 
contains no brick objects.

\subsection*{Acknowledgments}

The authors thank the referee for his/her very careful reading
and valuable comments and thank Professor Jarod Alper for many 
useful conversations on the subject. J. Chen was partially 
supported by the National Natural Science Foundation of China 
(Grant Nos. 11971398 and 12131018) and the Fundamental Research 
Funds for Central Universities of China (Grant No. 20720180002). 
Z. Gao was partially supported by the National Natural Science 
Foundation of China (Grant No. 61971365). E. Wicks and J.J. Zhang 
were partially supported by the US National Science Foundation 
(Grant Nos. DMS-1402863, DMS-1700825 and DMS-2001015). X.-H. Zhang 
was partially supported by the National Natural Science Foundation 
of China (Grant No. 11401328). H. Zhu was partially supported by a 
grant from Jiangsu overseas Research and Training Program for 
university prominent young and middle-aged Teachers and Presidents, 
China.



\end{document}